%% file: paper_arxiv.tex
\documentclass[11pt]{article}

\usepackage{amsmath,amssymb,amsfonts,textcomp,amsthm,xifthen,psfrag,graphicx,color,MnSymbol}
\usepackage[T1]{fontenc}

\usepackage{hyperref}

\date{}

\input{header}

\title{Trace operators of the bi-Laplacian and applications
\thanks{Supported by CONICYT through FONDECYT projects 11170050, 1190009.}}
\author{
Thomas~F\"uhrer\thanks{
Facultad de Matem\'aticas, Pontificia Universidad Cat\'olica de Chile,
Avenida Vicu\~na Mackenna 4860, Santiago, Chile,
email: {\tt \{tofuhrer,nheuer\}@mat.uc.cl}}
\and
Alexander Haberl\thanks{
Institute for Analysis and Scientific Computing,
Technische Universit\"at Wien, Austria,
email: {\tt alexander.haberl@asc.tuwien.ac.at}}
\and
Norbert Heuer$^\dagger$
}

\begin{document}
\maketitle
\begin{abstract}

We study several trace operators and spaces that are related to the bi-Laplacian.
They are motivated by the development of ultraweak formulations for the bi-Laplace equation with
homogeneous Dirichlet condition, but are also relevant to describe conformity of mixed approximations.

Our aim is to have well-posed (ultraweak) formulations that assume low regularity,
under the condition of an $L_2$ right-hand side function.
We pursue two ways of defining traces and corresponding integration-by-parts formulas.
In one case one obtains a non-closed space. This can be fixed by switching to the
Kirchhoff--Love traces from
[F\"uhrer, Heuer, Niemi,
 An ultraweak formulation of the Kirchhoff--Love plate bending model and DPG approximation,
 {\em Math. Comp.}, 88 (2019)].
Using different combinations of trace operators we obtain two well-posed formulations.
For both of them we report on numerical experiments with the DPG method and optimal test functions.

In this paper we consider two and three space dimensions. However, with the exception of
a given counterexample in an appendix (related to the non-closedness of a trace space), our
analysis applies to any space dimension larger than or equal to two.


\end{abstract}

\section{Introduction}

The bi-Laplace operator and biharmonic functions have generated sustained interest in the mathematics
community until today. Just in numerical analysis, MathSciNet reports well beyond 500 publications
with these key words in their titles. An early overview of numerical methods for the Dirichlet
problem of the bi-Laplacian is given by Glowinski and Pironneau in \cite{GlowinskiP_79_NMF}.
A more recent discussion can be found in the introduction of \cite{CockburnDG_09_HSD}.

Our interest in this operator arose while studying the Kirchhoff--Love plate bending model
and its numerical approximation by the discontinuous Petrov--Galerkin method with optimal test
functions (DPG method). It is well known that the Kirchhoff--Love model (with constant coefficients)
reduces to the bi-Laplace equation when considering the deflection of the plate as the only unknown.

In this paper we introduce and analyze trace operators that stem from the bi-Laplacian
and relate to integration-by-parts formulas. Such operators are of general interest as they
characterize interface conditions for (piecewise) sufficiently smooth functions to be globally
in the domain of the bi-Laplacian, or the subordinated Laplacian when considering the Laplacian
of the unknown as independent unknown.
Specifically, this analysis is required to construct conforming
finite element spaces of minimal regularity.
Regularity is a delicate issue when splitting the bi-Laplace equation into two Laplace equations
(explicitly, or implicitly through a mixed formulation).
Early papers on this technique are by Ciarlet and Raviart \cite{CiarletR_74_MFE},
and Monk \cite{Monk_88_IFE}. Regularity issues at corners have been analyzed, e.g., in
\cite{GerasimovSS_12_CGP,DeCosterNS_15_SBD}.
Thus, our aim is to use the least possible regularity subject to a given right-hand side
function in $L_2$. We note that Zulehner \cite{Zulehner_15_CRM} presents a formulation
(and space) where less regular right-hand side functions are permitted.

We consider Dirichlet boundary conditions, that is, a clamped plate
in the two-dimensional case. Here we only note that, in principle, it is possible to study
different boundary conditions, but the regularity of solutions will depend on them and some technical
details can be tricky.

In the rest of this paper we motivate our definitions and analysis by requirements for
the DPG method. For instance, the right-hand side function to be in $L_2$ is such a requirement.
Considering this method, there are good reasons to use ultraweak variational formulations.
From the mathematical point of view they simplify the analysis of well-posedness as they
allow for exact representations of adjoint operators, cf.~\cite{DemkowiczG_11_ADM}.
From a practical point of view they give access to approximations of field variables
that are close to optimal in the $L_2$ sense, cf., e.g.,~\cite{DemkowiczH_13_RDM,HeuerK_17_RDM}
for singularly perturbed problems. For general second order elliptic problems,
the $L_2$-optimality up to higher order terms is proved in \cite{Fuehrer_SDM}.
Now, since field variables of ultraweak formulations are only $L_2$-elements, the inherent regularity
of the underlying problem is passed onto appearing traces. Therefore, the study of trace spaces
is at the heart of proving well-posedness of ultraweak formulations. As explained before, the appearing
traces (and trace spaces) are equally relevant for the underlying problem and other variational
formulations as they precisely describe the notion of conformity and represent tools for its study.

It is the nature of DPG methods to use product test spaces (defined on meshes).
This is a fundamental paradigm proposed by Demkowicz and Gopalakrishnan in \cite{DemkowiczG_11_CDP}.
For that reason, our traces will live in product spaces related to the boundaries of elements.
Nevertheless, our results will apply to operations on domains without mesh, simply by using
meshes that consist of a single element.

The remainder of this paper is as follows.
In the next section we fix our model problem and introduce a setting needed to develop
ultraweak variational formulations. This approach motivates the framework in which we study
trace and jump operators, and trace spaces, and is presented in \S\ref{sec_traces_jumps}.
Aiming at lowest regularity, we first develop a setting where the unknown $u$ of the bi-Laplace equation
and its Laplacian (as independent unknown) are considered as elements of the same regularity
($L_2$-functions whose Laplacian is in $L_2$). This is done in \S\ref{sec_trace1}.
Later, in \S\ref{sec_VF2}, we present a variational formulation based on this framework,
state its well-posedness and equivalence with the model problem (Theorem~\ref{thm_stab1}),
and prove the quasi-optimal convergence of the induced DPG scheme (Theorem~\ref{thm_DPG1}).
For this formulation, discrete subspaces with good approximation properties seem
to require coupled basis functions (trace components are not independent).
This limits the practicality of the induced DPG scheme.
We therefore also consider the option of using more regular test functions (then trace components
can be approximated separately). This change gives rise to different trace operators
($\traceDt{}$ acting on $u$, and $\tracetD{}$ acting on $\sigma=\Delta u$) and spaces.
They are studied in \S\ref{sec_trace2}. Unfortunately, it turns out that the image of
$\tracetD{}$ is not closed (this is proved in Appendix~\ref{sec_app1}).
We therefore embed this space in a larger, closed trace space known from our
Kirchhoff--Love traces studied in \cite{FuehrerHN_19_UFK}.
The corresponding variational formulation and DPG scheme are presented in \S\ref{sec_VF2},
stating well-posedness and quasi-optimal convergence by Theorems~\ref{thm_stab2} and~\ref{thm_DPG2},
respectively. Proofs of Theorems~\ref{thm_stab1} and~\ref{thm_stab2} are given in \S\ref{sec_proofs}.
We do not provide a discrete analysis here, but we do present some numerical experiments in \S\ref{sec_num}.
They illustrate expected convergence properties.

One conclusion of our analysis is that the solution $u$ to the bi-Laplace
equation with right-hand side function in $L_2$ and homogeneous boundary condition satisfies $u\in H^2$.
Since we have not seen this result in the literature, we resume and prove this statement
in Appendix~\ref{sec_reg}.
There are, however, related $H^2$-regularity results for the bi-Laplacian by Girault and Raviart
in \cite[\S{5}]{GiraultR_86_FEM} (dimensions $2$ and $3$),
and by De~Coster \emph{et al.} \cite{DeCosterNS_15_SBD} for corner-type domains in $\R^2$.

Throughout the paper, $a\lesssim b$ means that $a\le cb$ with a generic constant $c>0$ that is independent of
the underlying mesh (except for possible general restrictions like shape-regularity of elements).
Similarly, we use the notation $a\simeq b$.

\section{Model problem} \label{sec_model}

Let $\Omega\subset\R^\di$ ($\di\in\{2,3\}$) be a bounded simply connected Lipschitz domain.
(We remark that our analysis and results will apply to any space dimension $\di\ge 2$,
with the exception of Lemma~\ref{la_trtD} with respect to the Dirac distributions and
the counterexample of Appendix~\ref{sec_app1}.
Nevertheless, we restrict ourselves to $\di\in\{2,3\}$ since we will make use of some results
from \cite{FuehrerHN_19_UFK} which are true in any space dimension $\di\ge 2$,
but are only formulated for $\di\in\{2,3\}$.)

The boundary of $\Omega$ is denoted by $\Gamma=\partial\Omega$ with exterior unit normal vector $\nn$.
For given $f\in L_2(\Omega)$ our model problem is
\begin{subequations} \label{prob}
\begin{alignat}{3}
    \Delta^2 u  &= f  && \quad\text{in}\ \Omega\label{p1},\\
    u = \partial_\nn u &= 0 && \quad\text{on}\ \Gamma.\label{p2}
\end{alignat}
\end{subequations}
We intend to develop an ultraweak formulation of \eqref{prob} with product test spaces.
To this end we consider a mesh $\cT$ that consists of general non-intersecting Lipschitz elements.
To the mesh $\cT=\{T\}$ we assign the skeleton $\cS=\{\partial T;\;T\in\cT\}$.

Introducing $\sigma:=\Delta u$,
we test the two equations $\Delta\sigma=f$, $\Delta u-\sigma=0$ on any $T\in\cT$ by
sufficiently smooth functions $v$ and $\tau$, respectively, and integrate by parts twice.
This formally gives
\[
    \vdual{\sigma}{\Delta v}_T
   +\dual{\partial_\nn\sigma}{v}_{\partial T}
   -\dual{\sigma}{\partial_\nn v}_{\partial T}
   +\vdual{u}{\Delta\tau}_T + \dual{\partial_\nn u}{\tau}_{\partial T} - \dual{u}{\partial_\nn\tau}_{\partial T}
   -\vdual{\sigma}{\tau}_T
     = \vdual{f}{v}_T,
\]
where $\vdual{\cdot}{\cdot}_T$ denotes the $L_2(T)$-duality. We still have to interpret
the dualities on $\partial T$ denoted by $\dual{\cdot}{\cdot}_{\partial T}$.
Summing over $T\in\cT$, we obtain, again formally,
\begin{align} \label{VFa}
   &\vdual{u}{\Delta \tau}_\cT + \vdual{\sigma}{\Delta v-\tau}_\cT
   \nonumber\\
   &
   + \sum_{T\in\cT} \dual{\partial_\nn\sigma}{v}_{\partial T}
   - \sum_{T\in\cT} \dual{\sigma}{\partial_\nn v}_{\partial T}
   + \sum_{T\in\cT} \dual{\partial_\nn u}{\tau}_{\partial T}
   - \sum_{T\in\cT} \dual{u}{\partial_\nn \tau}_{\partial T}
   = \vdual{f}{v}.
\end{align}
Here and in the following, $\vdual{\cdot}{\cdot}_\cT$ denotes the $L_2$-duality in the product space $L_2(\cT)$,
meaning that appearing differential operators are taken piecewise with respect to $T\in\cT$.
Below, we also use the notation of differential operators with index $\cT$ to indicate piecewise operations, e.g.,
$\vdual{\pwDelta u}{v}=\vdual{\Delta u}{v}_\cT$.
Furthermore, from now on, $\nn$ denotes a generic unit normal vector on $\partial T$ ($T\in\cT$) and $\Gamma$,
pointing outside $T$ and $\Omega$, respectively.

Before returning to our formulation \eqref{VFa} we need to study trace operators
to give a meaning to the skeleton dualities appearing in \eqref{VFa}.
This will be done next, before returning to \eqref{VFa} in \S\ref{sec_VF1}, and again in \S\ref{sec_VF2}.

\section{Traces and jumps} \label{sec_traces_jumps}

\subsection{Spaces and norms}

Given $T\in\cT$, and sufficiently smooth scalar (respectively, symmetric tensor) function $z:\;T\to\R$
(respectively, $\TTheta:\;T\to\R^{d\times d}$), we define the norms
$\|\cdot\|_{\Delta,T}$, $\|\cdot\|_{2,T}$ and $\|\cdot\|_\trddiv{T}$ by
\begin{align*}
   \|z\|_{\Delta,T}^2 &:= \|z\|_T^2 + \|\Delta z\|_T^2,\quad
   \|z\|_{2,T}^2 := \|z\|_T^2 + \|\Grad\grad z\|_T^2,\quad
   \|\TTheta\|_\trddiv{T}^2 := \|\TTheta\|_T^2 + \|\div\Div\TTheta\|_T^2.
\end{align*}
Here, $\|\cdot\|_T$ is the $L_2(T)$-norm (for scalar and tensor-valued functions),
$\Grad(\cdot):=\frac 12(\grad(\cdot)+\grad(\cdot)^\transp)$ denotes the symmetric gradient,
that is, $\Grad\grad z$ is the Hessian of $z$, $\div\!$ is the standard divergence operator, and
$\Div\!$ is the divergence applied row-wise to tensors.
Analogously, we use the corresponding norms on $\Omega$
where we drop the index $T$. For instance, $\|\cdot\|$ is the $L_2(\Omega)$-norm.
We also need the $L_2(\Omega)$-bilinear form $\vdual{\cdot}{\cdot}$, for scalar and tensor functions.

We define the spaces $\HD{T}$ and $H^2(T)$ as the closures of $\cD(\overline{T})$ with respect to
the norms $\|\cdot\|_{\Delta,T}$ and $\|\cdot\|_{2,T}$, respectively. Correspondingly,
$\HdDiv{T}$ is the closure of the space of smooth symmetric tensors on $T$ with respect to
$\|\cdot\|_\trddiv{T}$.
Analogously, $\HD{\Omega}$ and $H^2_0(\Omega)$ are the respective closures of $\cD(\overline{\Omega})$
and $\cD(\Omega)$, with norms $\|\cdot\|_{\Delta}$ and $\|\cdot\|_2$,
and $\HdDiv{\Omega}$ is the closure with respect to $\|\cdot\|_\mathrm{dDiv}$
of the space of smooth symmetric tensors on $\Omega$.

Given the mesh $\cT$, we will need the induced product spaces
\begin{align*}
   \HD{\cT} &:= \{z\in L_2(\Omega);\; z|_T\in \HD{T}\ \forall T\in\cT\},\\
   H^2(\cT) &:= \{z\in L_2(\Omega);\; z|_T\in H^2(T)\ \forall T\in\cT\},\\
   \HdDiv{\cT} &:= \{\TTheta\in \LL_2^s(\Omega);\; \TTheta|_T\in\HdDiv{T}\ \forall T\in\cT\}
\end{align*}
with canonical product norms $\|\cdot\|_{\Delta,\cT}$, $\|\cdot\|_{2,\cT}$, and $\|\cdot\|_\trddiv{\cT}$,
respectively. Here, $\LL_2^s$ indicates the space of symmetric $L_2$-tensors on the indicated domain.

\subsection{Traces and jumps, part one} \label{sec_trace1}

We define linear operators $\traceD{T}:\;\HD{T}\to \HD{T}'$ for $T\in\cT$ by
\begin{align} \label{trDT}
   \dual{\traceD{T}(z)}{v}_{\partial T} := \vdual{\Delta v}{z}_T - \vdual{v}{\Delta z}_T
   \quad\forall v\in \HD{T},
\end{align}
and observe that, for sufficiently smooth functions $v$ and $z$,
\begin{equation} \label{trDT_classical}
   \dual{\traceD{T}(z)}{v}_{\partial T}
   =
   \dual{z}{\partial_\nn v}_{\partial T} - \dual{v}{\partial_\nn z}_{\partial T}
\end{equation}
with $L_2(\partial T)$-duality $\dual{\cdot}{\cdot}_{\partial T}$ and standard trace and
normal derivative.
In other words, the trace operator $\traceD{T}$ can deliver standard traces
(trace and normal derivative) on $\partial T$ when diverting from the setting as a map
from $\HD{T}$ to its dual. This will be further discussed in \S\ref{sec_trace2} below.
Note the duality
\[
   \dual{\traceD{T}(z)}{v}_{\partial T} = - \dual{\traceD{T}(v)}{z}_{\partial T}
   \quad\forall z,v\in\HD{T}.
\]
The range of $\traceD{T}$ is
\[
   \bHD{\partial T} := \traceD{T}(\HD{T}\quad (T\in\cT).
\]
Switching from individual elements $T\in\cT$ to the whole of $\cT$, a collective trace operator
is defined by
\[
   \traceD{}:\;
   \left\{\begin{array}{cll}
      \HD{\Omega} & \to & \HD{\cT}',\\
      z & \mapsto & \traceD{}(z) := (\traceD{T}(z))_T
   \end{array}\right.,
\]
with duality
\begin{align}
    \label{trD_duality}
    \dual{\traceD{}(z)}{v}_\cS &:= \sum_{T\in\cT} \dual{\traceD{T}(z)}{v}_{\partial T}
    \quad(z\in\HD{\Omega},\ v\in \HD{\cT}).
\end{align}
To define a trace space that reflects the homogeneous boundary condition under consideration,
we make use of the operator $\traceD{\Omega}$ that is defined like $\traceD{T}$ by replacing $T$ with $\Omega$:
\begin{align} \label{trace_Omega}
   \dual{\traceD{\Omega}(z)}{v}_\Gamma := \vdual{z}{\Delta v} - \vdual{\Delta z}{v}
   \quad (z,v\in\HD{\Omega}).
\end{align}
Then, with
\[
   \HDz{\Omega} := \ker(\traceD{\Omega}),
\]
we introduce the product trace spaces
\[
   \bHDzz{\cS}  := \traceD{}(\HDz{\Omega})
   \ \subset\ 
   \bHD{\cS}    := \traceD{}(\HD{\Omega})
   \ \subset\ 
   \HD{\cT}'.
\]
Below, we refer to elements of such skeleton trace spaces in the form, e.g., $\tv=(\tv_T)_{T\in\cT}$.

The local and global trace spaces are equipped with the canonical trace norms,
\begin{align*}
   &\|\tv\|_\trD{\partial T} = \inf\{\|v\|_{\Delta,T};\; v\in \HD{T},\ \traceD{T}(v)=\tv\}
   \quad (\tv\in\bHD{\partial T},\ T\in\cT),
   \\
   &\|\tv\|_\trD{\cS} = \inf\{\|v\|_\Delta;\; v\in \HD{\Omega},\ \traceD{}(v)=\tv\}
   \qquad (\tv\in\bHD{\cS}\cup\bHDzz{\cS}).
\end{align*}
(Obviously, $\bHDzz{\cS}$ is a subspace of $\bHD{\cS}$. But here, and in some instances below,
we write $\bHD{\cS}\cup\bHDzz{\cS}$ to stress the fact that both spaces are furnished with
the same norm.)
Alternative norms are defined by duality,
\begin{align*}
   \|\tv\|_{\Delta',\partial T} &:= 
   \sup_{0\not=z\in\HD{T}} \frac{\dual{\tv}{z}_{\partial T}}{\|z\|_{\Delta,T}}
   \quad (\tv\in \bHD{\partial T},\ T\in\cT),
   \\
   \|\tv\|_{\Delta',\cS} &:= 
   \sup_{0\not=z\in\HD{\cT}} \frac{\dual{\tv}{z}_{\cS}}{\|z\|_{\Delta,\cT}}
   \quad (\tv\in \bHD{\cS}\cup\bHDzz{\cS}).
\end{align*}
Here, the dualities on $\partial T$ and $\cS$ are given by the corresponding trace
operations, \eqref{trDT} for the local spaces and \eqref{trD_duality} on $\cS$.
For instance, the duality between $\tv\in\bHD{\partial T}$ and $z\in\HD{T}$ is
$\dual{\tv}{z}_{\partial T}=\vdual{\Delta z}{v}_T-\vdual{z}{\Delta v}_T$ with arbitrary
$v\in\HD{T}$ such that $\traceD{T}(v)=\tv$.

\begin{lemma} \label{la_tr_unity}
It holds the identity
\[
   \|\tz\|_{\Delta',\partial T} = \|\tz\|_\trD{\partial T}\quad
   \forall \tz\in \bHD{\partial T},\ T\in \cT,
\]
so that
\[
   \traceD{T}:\; \HD{T}\to \bHD{\partial T}
\]
has unit norm and $(\bHD{\partial T},\|\cdot\|_{\Delta',\partial T})$ is closed.
\end{lemma}

\begin{proof}
The proof is essentially identical to the one of Lemma~3.2 in \cite{FuehrerHN_19_UFK}. We just need to replace
spaces, operators and norms by the ones used here. For the convenience of the reader we repeat the proof.

The estimate $\|\tz\|_{\Delta',\partial T}\le \|\tz\|_\trD{\partial T}$ is due to the boundedness
\begin{align*}
   \dual{\traceD{T}(z)}{v}_{\partial T} &\le \|z\|_{\Delta,T} \|v\|_{\Delta,T}
   \quad\forall z,v\in\HD{T},\ T\in\cT.
\end{align*}
To show the other direction we consider an element $T\in\cT$ and $\tz\in\bHD{\partial T}$,
and define $v\in\HD{T}$ by solving
\begin{align} \label{prob_dd_z}
   \vdual{\Delta v}{\Delta \deltav}_T + \vdual{v}{\deltav}_T
   =
   \dual{\tz}{\deltav}_{\partial T}
   \quad\forall \deltav\in\HD{T}.
\end{align}
One deduces that
\begin{align} \label{pde_dd_z}
   \Delta^2 v + v = 0\quad\text{in}\ L_2(T).
\end{align}
We then define $z\in\HD{T}$ as the solution to
\begin{align} \label{prob_dd_QQ}
   \vdual{\Delta z}{\Delta\deltaz}_T + \vdual{z}{\deltaz}_T
   =
   \dual{\traceD{T}(\deltaz)}{v}_{\partial T}
   \quad\forall\deltaz\in \HD{T}.
\end{align}
Again, it holds
\begin{align} \label{pde_dd_QQ}
   \Delta^2 z + z = 0\quad\text{in}\ L_2(T).
\end{align}
Let us show that $z=\Delta v$. To this end we define $z^*:=\Delta v$ and find that
\(
   \Delta z^* = -v,
\)
cf.~\eqref{pde_dd_z}.
Using this relation, and the definitions of $z^*$ and $\traceD{T}$, cf.~\eqref{trDT}, we obtain
\begin{align*}
   \vdual{\Delta z^*}{\Delta\deltaz}_T + \vdual{z^*}{\deltaz}_T
   &=
   -\vdual{v}{\Delta\deltaz}_T + \vdual{\Delta v}{\deltaz}_T
   =
   \dual{\traceD{T}(\deltaz)}{v}_{\partial T}
\end{align*}
for any $\deltaz\in\HD{T}$. This shows that $z^*$ solves \eqref{prob_dd_QQ}, that is,
$z=z^*=\Delta v$.
Due to this relation and $\Delta z = -v$, it follows by \eqref{prob_dd_z} that
\begin{align*}
   \dual{\traceD{T}(z)}{\deltav}_{\partial T}
   &=
   \vdual{z}{\Delta\deltav}_T - \vdual{\Delta z}{\deltav}_T
   \\
   &=
   \vdual{\Delta v}{\Delta\deltav}_T + \vdual{v}{\deltav}_T
   =
   \dual{\tz}{\deltav}_{\partial T}
   \quad\forall\deltav\in\HD{T}.
\end{align*}
In other words, $\traceD{T}(z)=\tz$.
This relation together with selecting $\deltav=v$ in \eqref{prob_dd_z} and $\deltaz=z$ in \eqref{prob_dd_QQ},
shows that
\begin{align*}
   \dual{\tz}{v}_{\partial T} = \|v\|_{\Delta,T}^2 = \|z\|_{\Delta,T}^2.
\end{align*}
Noting that
\(
   \|z\|_{\Delta,T} = \|\tz\|_\trD{\partial T}
\)
by \eqref{pde_dd_QQ}, this relation finishes the proof of the norm identity.
The space $\bHD{\partial T}$ is closed as the image of a bounded below operator.
\end{proof}

\begin{prop} \label{prop_D_jump}
(i) For $z\in\HD{\cT}$ it holds
\[
   z\in\HD{\Omega} \quad\Leftrightarrow\quad
   \dual{\traceD{}(v)}{z}_\cS = 0\quad\forall v\in\HDz{\Omega}
\]
and
\[
   z\in\HDz{\Omega} \quad\Leftrightarrow\quad
   \dual{\traceD{}(v)}{z}_\cS = 0\quad\forall v\in\HD{\Omega}.
\]
(ii) The identity
\begin{align*}
   \sum_{T\in\cT} \|\tz_T\|_\trD{\partial T}^2 = \|\tz\|_\trD{\cS}^2
   \quad\forall \tz=(\tz_T)_T \in \bHD{\cS}\cup\bHDzz{\cS}
\end{align*}
holds true.
\end{prop}

\begin{proof}
The proof of (i) follows the standard procedure, cf.~\cite[Proof of Theorem 2.3]{CarstensenDG_16_BSF}
and \cite[Proof of Proposition 3.8(i)]{FuehrerHN_19_UFK}.
For $z\in\HD{\Omega}$ and $v\in\HDz{\Omega}$ we have that
\begin{align*}
   -\dual{\traceD{}(z)}{v}_\cS
   =
   \dual{\traceD{}(v)}{z}_\cS
   &\overset{\mathrm{def}}=
   \sum_{T\in\cT} \vdual{\Delta z}{v}_T - \vdual{z}{\Delta v}_T
   =
   \vdual{\Delta z}{v} - \vdual{z}{\Delta v} = \dual{\traceD{\Omega}(v)}{z}_\Gamma = 0.
\end{align*}
The penultimate step is due to \eqref{trace_Omega}, and the last identity holds since
$\traceD{\Omega}(v)=0$ by definition of $\HDz{\Omega}$.
This is the direction ``$\Rightarrow$'' in both statements of part (i).

Now, for given $z\in\HD{\cT}$ with $\dual{\traceD{}(v)}{z}_\cS = 0$ for any $v\in\HDz{\Omega}$
we have in the distributional sense
\[
   \Delta z(v)=\vdual{z}{\Delta v}=\vdual{\Delta z}{v}_\cT - \dual{\traceD{}(v)}{z}_\cS
   = \vdual{\Delta z}{v}_\cT\quad\forall v\in\mathcal{D}(\Omega).
\]
Therefore, $\Delta z\in L_2(\Omega)$, that is, $z\in\HD{\Omega}$.

Analogously, if $\dual{\traceD{}(v)}{z}_\cS = 0$ for any $v\in\HD{\Omega}$,
we conclude as before that $z\in\HD{\Omega}$. Then,
\[
   0 = \dual{\traceD{}(v)}{z}_\cS = \vdual{v}{\Delta z} - \vdual{\Delta v}{z}
     = -\dual{\traceD{\Omega}(z)}{v}_\Gamma
   \quad\forall v\in\HD{\Omega}
\]
implies that $\traceD{\Omega}(z)=0$, cf.~\eqref{trace_Omega}. That is, $z\in\HDz{\Omega}$.

It remains to prove (ii). Here we follow \cite[Proof of Proposition 3.8(ii)]{FuehrerHN_19_UFK}.
By definition of the norms it holds $\sum_{T\in\cT} \|\tz_T\|_\trD{\partial T}^2 \le \|\tz\|_\trD{\cS}^2$
for any $\tz=(\tz_T)_T\in\bHD{\cS}\cup\bHDzz{\cS}$.
To show the other bound let
$\tz=(\tz_T)_T\in \bHD{\cS}\cup\bHDzz{\cS}$ be given with $z\in\HD{\Omega}$ such that $\traceD{}(z)=\tz$.
Furthermore, for any $T\in\cT$, there exists $\tilde z_T\in\HD{T}$ such that $\traceD{T}(\tilde z_T)=\tz_T$
and $\|\tilde z_T\|_{\Delta,T}=\|\tz_T\|_\trD{\partial T}$.
Defining $\tilde z\in\HD{\cT}$ by $\tilde z|_T:=\tilde z_T$ ($T\in\cT$) we find with part (i) that
$\tilde z\in\HD{\Omega}$ with $\traceD{}(\tilde z)=\tz$.
Therefore,
\begin{align*}
   \sum_{T\in\cT} \|\tz_T\|_\trD{\partial T}^2
   &=
   \sum_{T\in\cT} \|\tilde z_T\|_{\Delta,T}^2 = \|\tilde z\|_\Delta^2
   \ge
   \|\tz\|_\trD{\cS}^2,
\end{align*}
which was left to prove.
\end{proof}

\begin{prop} \label{prop_D_trace}
It holds the identity
\[
   \|\tz\|_{\Delta',\cS} = \|\tz\|_\trD{\cS}
   \quad\forall\tz\in\bHD{\cS}.
\]
In particular,
\[
   \traceD{}:\; \HD{\Omega}\to \bHD{\cS},\qquad
   \traceD{}:\; \HDz{\Omega}\to \bHDzz{\cS}
\]
have unit norm and $\bHD{\cS}$, $\bHDzz{\cS}$ are closed.
\end{prop}

\begin{proof}
Having the tools at hand, the proof is standard (cf., e.g., \cite[Theorem~2.3]{CarstensenDG_16_BSF}
and \cite[Proposition 3.5]{FuehrerHN_19_UFK}).
By definition of the involved norms, a duality argument in product spaces, Lemma~\ref{la_tr_unity}
and Proposition~\eqref{prop_D_jump}(ii) one finds that
\begin{align*}
   \|\tz\|_{\Delta',\cS}^2
   &=
   \Bigl(\sup_{0\not=v\in\HD{\cT}} \frac{\sum_{T\in\cT}\dual{\tz_T}{v}_{\partial T}}{\|v\|_{\Delta,\cT}}\Bigr)^2
   =
   \sum_{T\in\cT}
   \sup_{0\not=v\in\HD{T}} \frac{\dual{\tz_T}{v}_{\partial T}^2}{\|v\|_{\Delta,T}^2}
   \\
   &=
   \sum_{T\in\cT} \|\tz_T\|_{\Delta',\partial T}^2
   =
   \sum_{T\in\cT} \|\tz_T\|_\trD{\partial T}^2
   =
   \|\tz\|_\trD{\cS}^2\qquad\forall\tz\in\bHD{\cS}.
\end{align*}
The spaces $\bHD{\cS}$ and $\bHDzz{\cS}$ are closed as the images of bounded below operators.
\end{proof}

\subsection{Traces and jumps, part two} \label{sec_trace2}

As it is not straightforward to discretize the range of $\traceD{T}$
(where the trace components are coupled), we proceed to introduce different trace operators and spaces.
According to the regularity of $u$ (the solution of \eqref{prob}) and $\Delta u$ (which will be
represented by an independent variable) we consider two different cases.

\subsubsection{Trace of $u$.}

Let us start by defining a trace operator that takes $H^2(T)$ instead of $\HD{T}$
as domain ($T\in\cT$). It is the restriction of $\traceD{T}$, cf.~\eqref{trDT},
\[
   \traceDt{T}:\;
   \left\{\begin{array}{cll}
      H^2(T) & \to & \HD{T}',\\
      v & \mapsto & \traceDt{T}(v):=\traceD{T}(v)
   \end{array}\right..
\]
Similarly as before, we have the duality relation
\[
   \dual{\traceDt{T}(v)}{z}_{\partial T} = - \dual{\traceD{T}(z)}{v}_{\partial T}
   \quad\forall v\in H^2(T),\ z\in\HD{T}.
\]
The corresponding collective trace operator (including boundary conditions) is
\[
   \traceDt{}:\;
   \left\{\begin{array}{cll}
      H^2_0(\Omega) & \to & \HD{\cT}',\\
      v & \mapsto & \traceDt{}(v) := (\traceDt{T}(v))_T
   \end{array}\right.
\]
with duality
\begin{align}
    \label{trDt_duality}
    \dual{\traceDt{}(v)}{z}_\cS &:= \sum_{T\in\cT} \dual{\traceDt{T}(v)}{z}_{\partial T}
    \quad(v\in H^2_0(\Omega),\ z\in\HD{\cT}).
\end{align}
The ranges of these operators are denoted by
\[
   \bHDt{\partial T} := \traceDt{T}(H^2(T)) \quad (T\in\cT)
   \quad\text{and}\quad
   \bHDtzz{\cS}      := \traceDt{}(H^2_0(\Omega)).
\]
As before, the local and global trace spaces are equipped with canonical trace norms,
\begin{align*}
   \|\tv\|_\trt{\partial T} &:= \inf\{\|v\|_{2,T};\; v\in H^2(T),\ \traceD{T}(v)=\tv\}
   &&(\tv\in\bHDt{\partial T},\ T\in\cT),\\
   \|\tv\|_\trt{\cS}        &:= \inf\{\|v\|_2;\; v\in H^2_0(\Omega),\ \traceD{}(v)=\tv\}
   &&(\tv\in\bHDtzz{\cS}),
\end{align*}
and alternative norms are induced by the respective duality,
\begin{align*}
   \|\tv\|_{\Delta',\partial T} &:= 
   \sup_{0\not=z\in\HD{T}} \frac{\dual{\tv}{z}_{\partial T}}{\|z\|_{\Delta,T}}
   && (\tv\in \bHDt{\partial T},\ T\in\cT),\\
   \|\tv\|_{\Delta',\cS} &:= 
   \sup_{0\not=z\in\HD{\cT}} \frac{\dual{\tv}{z}_{\cS}}{\|z\|_{\Delta,\cT}}
   && (\tv\in \bHDtzz{\cS}).
\end{align*}
It goes without saying that the dualities on $\partial T$ and $\cS$ are defined by the corresponding trace
operations \eqref{trDT} (generically for any local space), and \eqref{trDt_duality} on $\cS$.
For instance, the duality $\dual{\tv}{z}_{\partial T}$ between
$\tv\in\bHDt{\partial T}$ and $z\in\HD{T}$ is
$\vdual{\Delta z}{v}_T-\vdual{z}{\Delta v}_T$ with arbitrary $v\in H^2(T)$ such that $\traceDt{T}(v)=\tv$.

It is immediate that all the trace operators are bounded both with respect to the respective canonical
trace norm and the respective duality norm.

\begin{remark} \label{rem_tr}
The trace operator $\traceDt{T}$ gives rise to two components,
$\traceDt{T}(v)=(v|_{\partial T},\partial_\nn v|_{\partial T})$ for $v\in H^2(T)$.
On a non-smooth boundary $\partial T$, they are generally not independent.
That is, this trace operator does not map surjectively onto the product space
of separate traces, $v|_{\partial T}$ and $\partial_\nn v|_{\partial T}$,
cf.~Grisvard~\cite{Grisvard_85_EPN}. In \cite{CostabelD_96_IBS}, Costabel and Dauge discuss
this subject including dual spaces.
\end{remark}

\begin{prop} \label{prop_Dt_trace}
It holds the identity
\[
   \|\tv\|_{\Delta',\cS} = \|\tv\|_\trt{\cS}
   \quad\forall\tv\in\bHDtzz{\cS}.
\]
In particular,
\[
   \traceDt{}:\; H^2_0(\Omega)\to \bHDtzz{\cS}
\]
has unit norm and $\bHDtzz{\cS}$ is closed.
\end{prop}

\begin{proof}
Let $\tv=(\tv_T)_T\in\bHDtzz{\cS}$ be given. By definition of the norms, one sees that
$\|\tv\|_{\Delta',\cS}\le \|v\|_\Delta$ for any $v\in H^2_0(\Omega)$ with $\traceDt{}(v)=\tv$. Since
\begin{equation} \label{eq_Dt}
   \|\Delta v\|=\|\Grad\grad v\| \quad\forall v\in H^2_0(\Omega)
\end{equation}
(cf.~\cite[(1.2.8)]{Ciarlet}) we conclude that
$\|\tv\|_{\Delta',\cS} \le \|\tv\|_\trt{\cS}$.

To show the other inequality, we define $v_T\in H^2(T)$ ($T\in\cT$) as the solution to
\[
   \bigl(\div\Div\Grad\grad v_T + v_T =\bigr)\ \Delta^2 v_T+v_T = 0\quad\text{in}\quad T,\qquad
   \traceDt{T}(v_T)=\tv_T,
\]
and introduce functions $v$, $z$ with $v|_T=v_T$ and $z|_T=\Delta v_T$ ($T\in\cT$).
We conclude that $v\in H^2_0(\Omega)$ and $\|v\|_2=\|\tv\|_\trt{\cS}$. Furthermore,
since $\Delta z_T=-v_T$, $z\in\HD{\cT}$, and also using relation \eqref{eq_Dt} we find that
\[
   \|z\|_{\Delta,\cT}^2 = \sum_{T\in\cT} \|\Delta v_T\|_T^2 + \|v_T\|_T^2 = \|v\|_\Delta^2=\|v\|_2^2.
\]
Finally, we observe that
\begin{align*}
   \|v\|_2^2 &= \|v\|_\Delta^2 = \sum_{T\in\cT} \vdual{\Delta v_T}{\Delta v_T}_T + \vdual{v_T}{v_T}_T
             = \sum_{T\in\cT} -\dual{\traceDt{T}(v_T)}{\Delta v_T}_{\partial T}
             = -\dual{\tv}{z}_\cS.
\end{align*}
Here, we made use of the relation $\Delta^2 v_T+v_T = 0$.
Collecting the findings we conclude that
\[
   \|\tv\|_\trt{\cS}^2 = \|v\|_2^2 = \|z\|_{\Delta,\cT}^2 = -\dual{\tv}{z}_\cS.
\]
This yields
\[
   \|\tv\|_\trt{\cS} \le \|\tv\|_{\Delta',\cS}
\]
and finishes the proof.
\end{proof}

\begin{remark} \label{rem_tr_local}
Comparing the results for our trace operators $\traceD{}$ (Proposition~\ref{prop_D_trace})
and $\traceDt{}$ (Proposition~\ref{prop_Dt_trace}) one notices that there is no result
for the local operator $\traceDt{T}$ that corresponds to Lemma~\ref{la_tr_unity}.
The reason for the lack of such a local property is that relation~\eqref{eq_Dt} requires
homogeneous boundary conditions.
\end{remark}

\begin{prop} \label{prop_tDt_jump}
For $z\in\HD{\cT}$ it holds
\[
   z\in\HD{\Omega} \quad\Leftrightarrow\quad
   \dual{\traceDt{}(v)}{z}_\cS = 0\quad\forall v\in H^2_0(\Omega).
\]
\end{prop}

\begin{proof}
The proof is analogous to that of Proposition~\ref{prop_D_jump}(i).
The direction ``$\Rightarrow$'' follows by integration by parts and
density arguments.
The other direction is proved by taking $z\in\HD{\cT}$
with $\dual{\traceDt{}(v)}{z}_\cS = 0$ for any $v\in H^2_0(\Omega)$,
and concluding that $\Delta z\in L_2(\Omega)$ so that $z\in\HD{\Omega}$.
\end{proof}

\subsubsection{Trace of $\Delta u$.}

Now let us turn to possible trace operations for $\Delta u$ ($u$ representing a function
with a regularity according to the solution of \eqref{prob}). Obviously,
since $f\in L_2(\Omega)$ by assumption, $\Delta u\in\HD{\Omega}$ by \eqref{p1}.
That is why we have considered the trace operator $\traceD{}$ in \S\ref{sec_trace1}.
Since we have restricted the domain for the definition of $\traceDt{}$, duality considerations
reveal that we now have to consider extended traces by testing with $H^2$-functions.
This seems to force to define an operator
\begin{equation} \label{trtDT}
   \tracetD{T}:\;
   \left\{\begin{array}{cll}
      \HD{T} & \to & H^2(T)',\\
      z & \mapsto & \tracetD{T}(z):=\traceD{T}(z)
   \end{array}\right.\quad (T\in\cT)
\end{equation}
with corresponding collective trace operator $\tracetD{}$, and trace norms and norms defined by duality
with $H^2$. Again, this operator gives rise to two components,
\[
   (\cdot)|_{\partial T}:\;\left\{\begin{array}{clc}
      \HD{T} & \to & \{z\in H^2(T);\; \partial_\nn z|_{\partial T}=0\}'\\
      v      & \mapsto & z\mapsto \dual{\tracetD{T}(v)}{z}_{\partial T}
   \end{array}\right.\qquad (T\in\cT)
\]
and
\[
   (\partial_\nn\,\cdot)|_{\partial T}:\;\left\{\begin{array}{clc}
      \HD{T} & \to & \{z\in H^2(T);\; z|_{\partial T}=0\}'\\
      v      & \mapsto & z\mapsto -\dual{\tracetD{T}(v)}{z}_{\partial T}
   \end{array}\right.\qquad (T\in\cT),
\]
cf.~\eqref{trDT_classical}.
For a smooth boundary $\partial T$, the two components are independent as in that case
the operator $\tracetD{T}$ maps $\HD{T}$ onto
$H^{-3/2}(\partial T)\times H^{-1/2}(\partial T):=H^{3/2}(\partial T)'\times H^{1/2}(\partial T)'$.
Here, $H^{3/2}(\partial T)$ denotes the space of traces onto $\partial T$ of $H^2(T)$-functions,
and $H^{1/2}(\partial T)$ is that of the normal derivatives.
Glowinski and Pironneau give details in \cite[Props 2.3, 2.4]{GlowinskiP_79_NMF} and refer
to Lions and Magenes for a proof, see~\cite[Chapter 2: Theorem 6.5, Section 9.8 (p.~213)]{LionsMagenes}.
However, on a polygonal element $T$, the trace operator is not surjective onto
$H^{-3/2}(\partial T)\times H^{-1/2}(\partial T)$. This has been indicated by Costabel and
Dauge in \cite{CostabelD_96_IBS}.
Furthermore, it turns out that in general the operator $\tracetD{T}$ is not bounded below.
We give a counterexample in the appendix.

For these reasons we avoid to employ the seemingly obvious choice \eqref{trtDT}.
Instead, we take a trace operator defined in \cite{FuehrerHN_19_UFK}. It can be interpreted
as an extension of $\tracetD{T}$ to a larger domain, see Lemma~\ref{la_trtD} below.
Let us repeat some definitions and needed properties from \cite{FuehrerHN_19_UFK}.

We introduce trace operators $\traceDD{T}:\;\HdDiv{T}\to H^2(T)'$
for $T\in\cT$ by
\begin{align} \label{trT_dd}
   \dual{\traceDD{T}(\TTheta)}{z}_{\partial T} := \vdual{\div\Div\TTheta}{z}_T - \vdual{\TTheta}{\Grad\grad z}_T,
\end{align}
with the collective variant defined as
\[
   \traceDD{}:\;
   \left\{\begin{array}{cll}
      \HdDiv{\Omega} & \to & H^2(\cT)',\\
      \TTheta & \mapsto & \traceDD{}(\TTheta) := (\traceDD{T}(\TTheta))_T
   \end{array}\right.
\]
with duality 
\begin{align} \label{tr_dd}
    \dual{\traceDD{}(\TTheta)}{z}_\cS := \sum_{T\in\cT} \dual{\traceDD{T}(\TTheta)}{z}_{\partial T}.
\end{align}
The range of $\traceDD{}$ is denoted by
\begin{align*}
   \bH^{-3/2,-1/2}(\cS) := \traceDD{}(\HdDiv{\Omega})
\end{align*}
and provided with the trace norm
\begin{align*}
   \|\tq\|_\trddiv{\cS} &:= \inf \Bigl\{\|\TTheta\|_{\div\Div};\; \TTheta\in \HdDiv{\Omega},\ \traceDD{}(\TTheta)=\tq\Bigr\}
\end{align*}
or the duality norm
\begin{align*}
   \|\tq\|_{-3/2,-1/2,\cS} &:=  \sup_{0\not=z\in H^2(\cT)} \frac{\dual{\tq}{z}_\cS}{\|z\|_{2,\cT}},
   \quad \tq\in \bH^{-3/2,-1/2}(\cS).
\end{align*}
Here, the duality is defined as
\begin{align} \label{tr_dd_dual}
   \dual{\tq}{z}_\cS := \sum_{T\in\cT} \dual{\tq_T}{z}_{\partial T}
\end{align}
with
\[
   \dual{\tq}{z}_{\partial T} := \dual{\traceDD{T}(\TTheta)}{z}_{\partial T}
   \quad\text{for } \TTheta\in \HdDiv{T}\text{ with } \traceDD{T}(\TTheta) = \tq=(\tq_T)_T,
\]
as in \eqref{trT_dd} and \eqref{tr_dd}.

\begin{prop}[{\cite[Proposition 5]{FuehrerHN_19_UFK}}] \label{prop_dd_trace}
It holds the identity
\[
   \|\tq\|_{-3/2,-1/2,\cS} = \|\tq\|_\trddiv{\cS}
   \quad\forall\tq\in\bH^{-3/2,-1/2}(\cS).
\]
In particular,
\[
   \traceDD{}:\; \HdDiv{\Omega}\to \bH^{-3/2,-1/2}(\cS)
\]
has unit norm and $\bH^{-3/2,-1/2}(\cS)$ is closed.
\end{prop}

\begin{prop}[{\cite[Proposition 8]{FuehrerHN_19_UFK}}] \label{prop_gg_jump}
For $z\in H^2(\cT)$ the following equivalence holds,
\[
   z\in H^2_0(\Omega)\quad\Leftrightarrow\quad
   \dual{\tq}{z}_\cS=0 \quad\forall\tq\in \bH^{-3/2,-1/2}(\cS).
\]
\end{prop}

Now, the connection between $\tracetD{}$ and $\traceDD{}$ is as follows.
For given $\sigma\in\HD{\Omega}$, it holds
$\diagtensor(\sigma):=\begin{pmatrix}\sigma & 0\\ 0 & \sigma\end{pmatrix}\in\HdDiv{\Omega}$
since $\div\Div\diagtensor(\sigma)=\Delta\sigma$, and one concludes that
\begin{align*}
   \dual{\tracetD{}(\sigma)}{v}_\cS
   &=
   \vdual{\sigma}{\Delta v}_\cT - \vdual{\Delta\sigma}{v}
   \\
   &=
   \vdual{\diagtensor(\sigma)}{\Grad\grad v}_\cT - \vdual{\div\Div\diagtensor(\sigma)}{v}
   =
   -\dual{\traceDD{}(\diagtensor(\sigma))}{v}_\cS
   \quad\forall v\in H^2(\cT).
\end{align*}
It is clear that $\diagtensor:\;\HD{\Omega}\to\HdDiv{\Omega}$ is not surjective.
Furthermore, since traces of images of $\diagtensor$ do not have jump terms at vertices of $\cT$
which are present in the case of traces of $\HdDiv{\Omega}$, see~\cite{FuehrerHN_19_UFK},
it is clear that $\tracetD{}$ does not map surjectively onto $\bH^{-3/2,-1/2}(\cS)$.

Let us note this result.

\begin{lemma} \label{la_trtD}
\[
   \tracetD{} = -\traceDD{}\circ\diagtensor:\; \HD{\Omega} \to \bH^{-3/2,-1/2}(\cS)
\]
is bounded but not surjective. In particular,
Dirac distributions at boundary points, $\delta_e:\;z\mapsto z|_T(e)$
($e\in\Gamma\cap\overline{T}$, $T\in\cT$, $z\in H^2(\cT)$ with $\mathrm{supp}(z)=\overline{T}$)
are elements of $\bH^{-3/2,-1/2}(\cS)$ but not of $\tracetD{}(\HD{\Omega})$.
\end{lemma}

The fact that $\delta_e\not\in\tracetD{}(\HD{\Omega})$ is illustrated in Appendix~\ref{sec_app2}.

\section{First variational formulation and DPG approximation} \label{sec_VF1}

Let us continue to develop a variational formulation of \eqref{prob}. Considering the
trace operator $\traceD{}$ from \S\ref{sec_trace1}, our preliminary formulation \eqref{VFa} now reads
\[
   \vdual{u}{\Delta \tau}_\cT + \vdual{\sigma}{\Delta v-\tau}_\cT
   - \dual{\traceD{}(\sigma)}{v}_\cS - \dual{\traceD{}(u)}{\tau}_\cS
   = \vdual{f}{v}.
\]
In this case, test functions $v$ and $\tau$ come from $\HD{\cT}$. Therefore, introducing
independent trace variables $\tsigma:=\traceD{}(\sigma)$, $\tu:=\traceD{}(u)$, and spaces
\[
   \UU_1 := L_2(\Omega)\times L_2(\Omega)\times \bHDzz{\cS} \times \bHD{\cS},\qquad
   \VV_1 := \HD{\cT}\times\HD{\cT}
\]
with respective norms
\begin{align*}
   \|(u,\sigma,\tu,\tsigma)\|_{\UU_1}^2
   &:=
   \|u\|^2 + \|\sigma\|^2 + \|\tu\|_\trD{\cS}^2 + \|\tsigma\|_\trD{\cS}^2,\qquad
   \|(v,\tau)\|_{\VV_1}^2
   :=
   \|v\|_{\Delta,\cT}^2 + \|\tau\|_{\Delta,\cT}^2,
\end{align*}
our first ultraweak variational formulation of \eqref{prob} is
\begin{align} \label{VF1}
   (u,\sigma,\tu,\tsigma)\in \UU_1:\quad
   b_1(u,\sigma,\tu,\tsigma;v,\tau) = L(v,\tau)
   \quad\forall (v,\tau)\in\VV_1,
\end{align}
in strong form written as $B_1(u,\sigma,\tu,\tsigma)=L\in\VV_1'$.
Here,
\begin{align} \label{b1}
   b_1(u,\sigma,\tu,\tsigma;v,\tau)
   :=
   \vdual{u}{\Delta \tau}_\cT + \vdual{\sigma}{\Delta v-\tau}_\cT
   - \dual{\tu}{\tau}_\cS
   - \dual{\tsigma}{v}_\cS,
\end{align}
\(
   L(v,\tau) := \vdual{f}{v},
\)
and $\dual{\cdot}{\cdot}_\cS$ refers to the duality between
$\bHD{\cS}$ (including $\bHDzz{\cS}$) and $\HD{\cT}$ implied by \eqref{trD_duality}.

\begin{theorem} \label{thm_stab1}
The operator $B_1:\;\UU_1\to\VV_1'$ is continuous and bounded below. In particular,
for any function $f\in L_2(\Omega)$, there exists
a unique and stable solution $(u,\sigma,\tu,\tsigma)\in \UU_1$ to \eqref{VF1},
\[
   \|u\| + \|\sigma\| + \|\tu\|_{\Delta,\cS} + \|\tsigma\|_{\Delta,\cS}
   \lesssim
   \|f\|
\]
with a hidden constant that is independent of $f$ and $\cT$.
Furthermore, \eqref{prob} and \eqref{VF1} are equivalent:
If $u\in H^2_0(\Omega)$ solves \eqref{prob} then
$\uu:=(u,\Delta u,\traceD{}(u),\traceD{}(\Delta u))$ solves \eqref{VF1};
and if $\uu=(u,\sigma,\tu,\tsigma)$ solves \eqref{VF1} then $u$ is element of $H^2_0(\Omega)$
and solves \eqref{prob}.
\end{theorem}

For a proof of this theorem we refer to Section~\ref{sec_proofs}.

A DPG approximation with optimal test functions based on formulation \eqref{VF1} is as follows.
We select discrete spaces $\UU_{1,h}\subset\UU$ and test spaces $\VV_{1,h}:=\ttt_1(\UU_{1,h})\subset\VV_1$ where
$\ttt_1:\;\UU_1\to\VV_1$ is the \emph{trial-to-test operator} defined by
\[
   \ip{\ttt_1(\uu)}{\vv}_{\VV_1} = b_1(\uu,\vv)\quad\forall\vv\in\VV_1.
\]
Here, $\ip{\cdot}{\cdot}_{\VV_1}$ is the inner product in $\VV_1$ that generates the selected norm
$\bigl(\|\cdot\|_{\Delta,\cT}^2+\|\cdot\|_{\Delta,\cT}^2\bigr)^{1/2}$.

Then, an approximation $\uu_h=(u_h,\sigma_h,\tu_h,\tsigma_h)\in\UU_{1,h}$ is defined as the solution to
\begin{align} \label{DPG1}
   b_1(\uu_h,\vv) = L(\vv) \quad\forall\vv\in\VV_{1,h}.
\end{align}
Being a minimum residual method it delivers the best approximation of the exact solution
in the residual norm $\|B_1(\cdot)\|_{\VV_1'}$, cf., e.g.,~\cite{DemkowiczG_11_ADM}.
Then, using the equivalence of the norms $\|B_1(\cdot)\|_{\VV_1'}$ and $\|\cdot\|_{\UU_1}$
stated by Theorem~\ref{thm_stab1}, we obtain its quasi-optimal convergence in the latter norm.

\begin{theorem} \label{thm_DPG1}
Let $f\in L_2(\Omega)$ be given and let $\uu$ be the solution of \eqref{VF1}.
For any finite-dimensional subspace $\UU_{1,h}\subset\UU_1$
there exists a unique solution $\uu_h\in\UU_{1,h}$ to \eqref{DPG1}. It satisfies the quasi-optimal
error estimate
\[
   \|\uu-\uu_h\|_{\UU_1} \lesssim \|\uu-\ww\|_{\UU_1}
   \quad\forall\ww\in\UU_{1,h}
\]
with a hidden constant that is independent of $f$, $\cT$ and $\UU_{1,h}$.
\end{theorem}


\section{Second variational formulation and DPG approximation} \label{sec_VF2}

Let us reconsider the preliminary formulation \eqref{VFa}.
We make use of the regularity $u\in H^2_0(\Omega)$. Then, the variable $\tu$ replaces
$\traceDt{}(u)\in\bHDtzz{\cS}$ instead of $\traceD{}(u)\in\bHDzz{\cS}$.
We then use test functions $v\in H^2(\cT)$ instead of $v\in\HD{\cT}$.
This means that we have a trace $\traceDD{}\circ\diagtensor(\Delta u)\in\bH^{-3/2,-1/2}(\cS)$
(cf.~Lemma~\ref{la_trtD}) rather than $\traceD{}(\Delta u)\in\bHD{\cS}$.
This corresponds to using the spaces
\[
   \UU_2 := L_2(\Omega)\times L_2(\Omega)\times \bHDtzz{\cS} \times \bH^{-3/2,-1/2}(\cS),\qquad
   \VV_2 := H^2(\cT)\times\HD{\cT}
\]
with respective norms
\begin{align*}
   \|(u,\sigma,\tu,\tsigma)\|_{\UU_2}^2
   &:=
   \|u\|^2 + \|\sigma\|^2 + \|\tu\|_{2,\cS}^2 + \|\tsigma\|_\trddiv{\cS}^2,\qquad
   \|(v,\tau)\|_{\VV_2}^2
   :=
   \|v\|_{2,\cT}^2 + \|\tau\|_{\Delta,\cT}^2.
\end{align*}
The corresponding ultraweak variational formulation is:
\begin{align} \label{VF2}
   (u,\sigma,\tu,\tsigma)\in \UU_2:\quad
   b_2(u,\sigma,\tu,\tsigma;v,\tau) = L(v,\tau)
   \quad\forall (v,\tau)\in\VV_2.
\end{align}
Here, the bilinear form $b_2$ is defined similarly as $b_1$ in \eqref{b1}, namely
\begin{align*}
   b_2(u,\sigma,\tu,\tsigma;v,\tau)
   :=
   \vdual{u}{\Delta \tau}_\cT + \vdual{\sigma}{\Delta v-\tau}_\cT
   - \dual{\tu}{\tau}_\cS
   + \dual{\tsigma}{v}_\cS.
\end{align*}
Specifically, the duality $\dual{\tu}{\tau}_\cS$ is the one induced by \eqref{trDt_duality}
analogously as before, and $\dual{\tsigma}{v}_\cS$ is defined by \eqref{tr_dd_dual}.
For consistency with trace definitions we have changed the sign in front of the latter
duality, cf.~Lemma~\ref{la_trtD}.

We refer to the strong form of \eqref{VF2} as $B_2(u,\sigma,\tu,\tsigma)=L\in\VV_2'$.

\begin{theorem} \label{thm_stab2}
The operator $B_2:\;\UU_2\to\VV_2'$ is continuous and bounded below. In particular,
for any function $f\in L_2(\Omega)$, there exists
a unique solution $(u,\sigma,\tu,\tsigma)$ of \eqref{VF2}. It holds the bound
\[
   \|u\| + \|\sigma\| + \|\tu\|_{2,\cS} + \|\tsigma\|_{\trddiv{\cS}} \lesssim \|f\|
\]
with a hidden constant that is independent of $f$ and $\cT$.
Furthermore, \eqref{prob} and \eqref{VF2} are equivalent:
If $u\in H^2_0(\Omega)$ solves \eqref{prob} then
$\uu:=(u,\Delta u,\traceDt{}(u),\traceDD{}(\diagtensor(\Delta u)))$ solves \eqref{VF2};
and if $\uu=(u,\sigma,\tu,\tsigma)$ solves \eqref{VF2} then $u\in H^2_0(\Omega)$ solves \eqref{prob}.
\end{theorem}

A proof of this result is given in Section~\ref{sec_proofs}.

The corresponding DPG approximation uses discrete spaces $\UU_{2,h}\subset\UU_2$
and test spaces $\VV_{2,h}:=\ttt_2(\UU_{2,h})\subset\VV_2$ where the trial-to-test operator
$\ttt_2:\;\UU_2\to\VV_2$ is defined by
\[
   \ip{\ttt_2(\uu)}{\vv}_{\VV_2} = b_2(\uu,\vv)\quad\forall\vv\in\VV_2
\]
with inner product $\ip{\cdot}{\cdot}_{\VV_2}$ that induces the norm
$\bigl(\|\cdot\|_{2,\cT}^2+\|\cdot\|_{\Delta,\cT}^2\bigr)^{1/2}$ in $\VV_2$.
The approximation $\uu_h=(u_h,\sigma_h,\tu_h,\tsigma_h)\in\UU_{2,h}$ is defined analogously as before,
\begin{align} \label{DPG2}
   b_2(\uu_h,\vv) = L(\vv) \quad\forall\vv\in\VV_{2,h}.
\end{align}
Again, this scheme converges quasi-optimally, see Theorem~\ref{thm_DPG1}.

\begin{theorem} \label{thm_DPG2}
Let $f\in L_2(\Omega)$ be given and let $\uu$ be the solution of \eqref{VF2}.
For any finite-dimensional subspace $\UU_{2,h}\subset\UU_2$
there exists a unique DPG approximation $\uu_h=(u_h,\sigma_h,\tu_h,\tsigma_h)\in\UU_{2,h}$
defined by \eqref{DPG2}.
It satisfies the quasi-optimal error estimate
\[
   \|\uu-\uu_h\|_{\UU_2} \lesssim \|\uu-\ww\|_{\UU_2} \quad\forall\ww\in\UU_{2,h}
\]
with a hidden constant that is independent of $f$, $\cT$ and $\UU_{2,h}$.
\end{theorem}

\section{Proofs of Theorems~\ref{thm_stab1} and~\ref{thm_stab2}} \label{sec_proofs}

We start by showing unique and stable solvability of the (self) adjoint problem to \eqref{prob},
with continuous spaces.

\begin{lemma} \label{la_adj2}
For given $g_1,g_2\in L_2(\Omega)$, there exists a unique solution $(v,\tau)\in H^2_0(\Omega)\times\HD{\Omega}$
of
\begin{subequations} \label{adj}
\begin{alignat}{3}
    \Delta v - \tau  &= g_1 &&\quad\text{in}\ \Omega,   \label{a1}\\
    \Delta\tau       &= g_2 &&\quad\text{in}\ \Omega.   \label{a2}
\end{alignat}
\end{subequations}
It satisfies
\[
   \|v\|_2 + \|\tau\|_\Delta \lesssim \|g_1\| + \|g_2\|
\]
with a constant that is independent of $g_1$, $g_2$ and $\cT$.
\end{lemma}

\begin{proof}
We write a variational formulation for $v$. Applying $\Delta$ to \eqref{a1} and using \eqref{a2},
this gives the relation
\[
   \Delta(\Delta v-g_1) = g_2\quad\text{in}\quad L_2(\Omega).
\]
Testing with $\deltav\in H^2_0(\Omega)$ and integrating by parts we see that $v\in H^2_0(\Omega)$ solves
\[
   \vdual{\Delta v}{\Delta\deltav}
   =
   \vdual{g_1}{\Delta\deltav} + \vdual{g_2}{\deltav}
   \quad\forall\deltav\in H^2_0(\Omega).
\]
By standard arguments, this problem has a unique solution with bound
\[
   \|\Grad\grad v\|^2 = \|\Delta v\|^2
   \le
   \Bigl(\|g_1\|^2 + \|g_2\|^2\Bigr)^{1/2}
   \|v\|_\Delta.
\]
Here, we made use of \eqref{eq_Dt}.
Using Poincar\'e's inequality $\|v\|\lesssim \|\Grad\grad v\|$ we conclude that
\[
   \|v\|_2 \lesssim \|g_1\| + \|g_2\|.
\]
A unique solution $(v,\tau)$ of \eqref{adj} is then obtained by setting $\tau:=\Delta v-g_1$,
with bound
\[
   \|\tau\| + \|\Delta\tau\|
   =
   \|\Delta v-g_1\| + \|g_2\|
   \lesssim
   \|g_1\| + \|g_2\|.
\]
This finishes the proof.
\end{proof}

\subsection{Proof of Theorem~\ref{thm_stab1}.}

\paragraph{Well-posedness of \eqref{VF1}.}
We check the standard conditions. The boundedness of $b_1$ and $L$ holds by definition of the norms
in $\UU_1$ and $\VV_1$.

The injectivity of the adjoint operator $B_1^*:\;\VV_1\to\UU_1'$ can be seen as follows.
Let $(v,\tau)\in\VV_1$ be such that $b_1(\uu;v,\tau)=0$ for any $\uu=(u,\sigma,\tu,\tsigma)\in\UU_1$.
The selection of $\uu=(0,0,\tu,0)$ for any $\tu\in\bHDzz{\cS}$ reveals that $\tau\in\HD{\Omega}$
by Proposition~\ref{prop_D_jump}(i). Analogously,
selecting $\uu=(0,0,0,\tsigma)$ with arbitrary $\tsigma\in\bHD{\cS}$, Proposition~\ref{prop_D_jump}(i)
shows that $v\in\HDz{\Omega}$. We conclude that $(v,\tau)\in\HDz{\Omega}\times\HD{\Omega}$ solves
$\tau=\Delta v$ and $\Delta\tau=0$. It follows that $\Delta^2 v = 0$. Since $v\in\HDz{\Omega}$,
so that $\traceD{\Omega}(v)=0$, relation \eqref{trace_Omega} shows that
$\|\Delta v\|^2=\vdual{\Delta^2 v}{v}=0$. In particular, $\tau=\Delta v=0$.
Now, defining $z\in H^2_0(\Omega)$ as the solution to $\Delta^2 z=v$, and again
using \eqref{trace_Omega}, we find that
\(
   \vdual{v}{v}=\vdual{v}{\Delta^2 z}=\vdual{\Delta v}{\Delta z}=0,
\)
that is, $v=0$.

It remains to check the inf--sup condition
\begin{equation} \label{infsup_D}
   \|B_1\uu\|_{\VV_1'} \gtrsim \|\uu\|_{\UU_1}\quad\forall\uu\in\UU_1.
\end{equation}
To this end we employ the technique proposed by Carstensen {\em et al.} in \cite{CarstensenDG_16_BSF}.
To simplify reading let use relate our notation to the one in \cite{CarstensenDG_16_BSF}.
\begin{align*}
   &X=\UU_1,\quad X_0 = L_2(\Omega)\times L_2(\Omega),\quad \hat X=\bHDzz{\cS}\times\bHD{\cS},\\
   &Y=\VV_1,\quad Y_0=\HDz{\Omega}\times\HD{\Omega},\quad b(\cdot,\cdot)=b_1(\cdot,\cdot),\\
   &b_0(x,y)=b_1(u,\sigma,0,0;v,\tau)=\vdual{u}{\Delta \tau}_\cT + \vdual{\sigma}{\Delta v-\tau}_\cT\quad
    \text{with $x=(u,\sigma)$, $y=(v,\tau)$},\\
   &\hat b(\hat x,y)=b_1(0,0,\tu,\tsigma;v,\tau)= -\dual{\tu}{\tau}_\cS - \dual{\tsigma}{v}_\cS\quad
    \text{with $\hat x=(\tu,\tsigma)$, $y=(v,\tau)$}.
\end{align*}
According to \cite[Theorem~3.3]{CarstensenDG_16_BSF} it suffices to show the two inf--sup properties
\begin{align}
   \label{infsup1}
   &\text{\cite[Ass.~3.1]{CarstensenDG_16_BSF}:}\
   \sup_{0\not=(v,\tau)\in\HDz{\Omega}\times\HD{\Omega}}
   \frac{b_1(u,\sigma,0,0;v,\tau)}{\|(v,\tau)\|_{\VV_1}}
   \gtrsim \|u\| + \|\sigma\|
   \quad\forall u, \sigma\in L_2(\Omega),
   \\
   \label{infsup2}
   &\text{\cite[(18)]{CarstensenDG_16_BSF}:}\qquad
   \sup_{0\not=(v,\tau)\in\VV_1}
   \frac{\dual{\tu}{\tau}_\cS + \dual{\tsigma}{v}_\cS}{\|(v,\tau)\|_{\VV_1}}
   \gtrsim
   \|\tu\|_\trD{\cS} + \|\tsigma\|_\trD{\cS}
   \nonumber\\
   &\hspace*{0.55\textwidth} \forall (\tu,\tsigma)\in \bHDzz{\cS}\times\bHD{\cS},
\end{align}
and the identity
\[
   \HDz{\Omega}\times\HD{\Omega}
   =
   \{(v,\tau)\in\VV_1;\; \dual{\tu}{\tau}_\cS + \dual{\tsigma}{v}_\cS = 0\
                         \forall (\tu,\tsigma)\in \bHDzz{\cS}\times\bHD{\cS}\}.
\]
This identity is true by Proposition~\ref{prop_D_jump}.
Lemma~\ref{la_adj2} shows that \eqref{infsup1} holds:
\begin{align} \label{B1_below}
   \Bigl(\|u\|^2 + \|\sigma\|^2\Bigr)^{1/2}
   &=
   \sup_{0\not=(g_1,g_2)\in L_2(\Omega)\times L_2(\Omega)}
   \frac {\vdual{u}{g_1} + \vdual{\sigma}{g_2}}
         {(\|g_1\|^2+\|g_2\|^2)^{1/2}}
   \nonumber\\
   &\lesssim
   \sup_{0\not=(v,\tau)\in H^2_0(\Omega)\times\HD{\Omega}}
   \frac {b_1(u,\sigma,0,0;v,\tau)}
         {(\|v\|_2^2+\|\tau\|_\Delta^2)^{1/2}}
   \\
   \nonumber
   &\le
   \sup_{0\not=(v,\tau)\in \HDz{\Omega}\times\HD{\Omega}}
   \frac {b_1(u,\sigma,0,0;v,\tau)}
         {\|(v,\tau)\|_{\VV_1}}
   \quad\forall u, \sigma\in L_2(\Omega).
\end{align}
Finally, Proposition~\ref{prop_D_trace} shows that
\eqref{infsup2} is satisfied. This finishes the proof of \eqref{infsup_D}, and of the theorem.

\paragraph{Equivalence of \eqref{prob} and \eqref{VF1}.}
By construction of \eqref{VF1}, any solution $u\in H^2_0(\Omega)$ of \eqref{prob} provides
a solution $\uu:=(u,\Delta u,\traceD{}(u),\traceD{}(\Delta u))\in\UU_1$ of \eqref{VF1}.
In fact, the regularity $u\in\HDz{\Omega}$, together with $f\in L_2(\Omega)$,
is sufficient for this conclusion.

To see the other direction we use that \eqref{VF1} is uniquely solvable.
Its solution $\uu=(u,\sigma,\tu,\tsigma)$ satisfies $u\in \HDz{\Omega}$ and solves $\Delta^2u = f$ in $\Omega$,
as can be seen as follows. Selecting smooth test functions $v$ and $\tau$ with supports on individual
elements, one obtains $\sigma=\pwDelta u$ and $\pwDelta\sigma=f$, first in the distributional sense
and then in $L_2(\Omega)$ by the regularity $\sigma,f\in L_2(\Omega)$.
Second, denoting as usual $\tu=(\tu_T)_T$, $\tsigma=(\tsigma_T)_T$,
and using test functions $v,\tau\in\cD(\overline{T})$ for $T\in\cT$, one concludes that
$\tu_T=\traceD{T}(u)$ and $\tsigma_T=\traceD{T}(\sigma)$ for any $T\in\cT$
so that $u\in\HDz{\Omega}$ and $\sigma\in\HD{\Omega}$ by Proposition~\ref{prop_D_jump}.
Altogether, $u\in\HDz{\Omega}$ solves $\Delta^2u=f$.
Since any such function $u$ leads to a solution of \eqref{VF1}, as noted before, one concludes
the stronger regularity $u\in H^2_0(\Omega)$ by uniqueness of \eqref{VF1}. Therefore,
$u\in H^2_0(\Omega)$ solves \eqref{prob}.

\subsection{Proof of Theorem~\ref{thm_stab2}.}

The proof of Theorem~\ref{thm_stab2} is analogous to the one of Theorem~\ref{thm_stab1}.
The equivalence between \eqref{prob} and \eqref{VF2} holds as before. To show the 
well-posedness of \eqref{VF2} we repeat the steps that show the well-posedness of \eqref{VF1}
where we only have to replace the corresponding ingredients. Specifically, the injectivity
of $B_2^*:\;\VV_2\to\UU_2'$ is obtained by using
Propositions~\ref{prop_tDt_jump} and~\ref{prop_gg_jump} instead of Proposition~\ref{prop_D_jump}(i)
to deduce the continuity $(v,\tau)\in H^2_0(\Omega)\times\HD{\Omega}$ of
$(v,\tau)\in\VV_2$ satisfying $b_2(\uu;v,\tau)=0$ $\forall\uu\in\UU_2$.
Then Lemma~\ref{la_adj2} shows that $(v,\tau)=0$.

The inf--sup condition for $B_2$, corresponding to \eqref{infsup_D}, is shown by
the same framework, based on the two inf--sup conditions
\begin{align}
   \label{infsup1b}
   &\sup_{0\not=(v,\tau)\in H^2_0(\Omega)\times\HD{\Omega}}
   \frac{b_2(u,\sigma,0,0;v,\tau)}{\|(v,\tau)\|_{\VV_2}}
   \gtrsim \|u\| + \|\sigma\|
   \quad\forall u, \sigma\in L_2(\Omega),
   \\
   \label{infsup2b}
   &\sup_{0\not=(v,\tau)\in\VV_2}
   \frac{\dual{\tu}{\tau}_\cS + \dual{\tsigma}{v}_\cS}{\|(v,\tau)\|_{\VV_2}}
   \gtrsim
   \|\tu\|_{2,\cS} + \|\tsigma\|_\trddiv{\cS}
   \quad\forall (\tu,\tsigma)\in \bHDtzz{\cS}\times\bH^{-3/2,-1/2}(\cS),
\end{align}
and the identity
\[
   H^2_0(\Omega)\times\HD{\Omega}
   =
   \{(v,\tau)\in\VV_2;\; -\dual{\tu}{\tau}_\cS + \dual{\tsigma}{v}_\cS = 0\
                         \forall (\tu,\tsigma)\in \bHDtzz{\cS}\times\bH^{-3/2,-1/2}(\cS)\}.
\]
This identity is true by Propositions~\ref{prop_tDt_jump} and~\ref{prop_gg_jump},
and \eqref{infsup1b} holds as we have seen with \eqref{B1_below}.
Finally, Propositions~\ref{prop_Dt_trace} and~\ref{prop_dd_trace} show that
\eqref{infsup2b} is satisfied.

\section{Numerical examples} \label{sec_num}

According to Theorems~\ref{thm_DPG1} and~\ref{thm_DPG2}, any conforming subspaces
$\UU_{1,h}\subset\UU_1$ and $\UU_{2,h}\subset\UU_2$ yield quasi-optimal approximations
$\uu_{1,h}\in\UU_{1,h}$ and $\uu_{2,h}\in\UU_{2,h}$, respectively, of the solution(s)
$\uu_1=(u,\Delta u,\traceD{}(u),\traceD{}(\Delta u))$ and
$\uu_2=(u,\Delta u,\traceDt{}(u),\traceDD{}(\diagtensor(\Delta u)))$
to \eqref{VF1} and \eqref{VF2}, respectively. (In fact, $\uu_1=\uu_2$.)
Here, $u\in H^2_0(\Omega)$ solves \eqref{prob},
and $\uu_{1,h}$ and $\uu_{2,h}$ are the solutions of \eqref{DPG1} and \eqref{DPG2}, respectively.

The construction of discrete subspaces of $\UU_1$ and $\UU_2$ and their approximation properties
is ongoing research. In the case of the Kirchhoff--Love model we have presented a fully discrete analysis
in \cite{FuehrerH_19_FDD}. Here, we only select some discrete spaces in an \emph{ad hoc} fashion and present
the corresponding convergence results without proving any convergence orders. Also, the construction
of appropriate Fortin operators (needed to take the approximation of optimal test functions into account)
is left open. Test functions are approximated by selecting identical meshes for ansatz and test
spaces, and increasing polynomial degrees in the test spaces (see~\cite{FuehrerH_19_FDD} for details).

Specifically, we consider the two-dimensional case $\di=2$, and
use regular triangular meshes $\cT$ of shape-regular elements, with mesh parameter
$h:=h_\cT := \max_{T\in\cT} \diam(T)$. The DPG method provides a built-in error estimator,
the residual norm $\eta:=\|B_i(\uu_i-\uu_{i,h})\|_{\VV_i'}$.
(We generically use $\eta$ and select $i=1$ or $i=2$ as needed.)
By the product form of the test spaces, $\eta$ is composed of local element contributions
\(
  \eta^2 = \sum_{T\in\cT} \eta(T)^2.
\)
For the case with singular solution we use these indicators to perform adaptive DPG schemes,
based on newest-vertex-bisection and D\"orfler marking with parameter of one half.

\subsection{Example with smooth solution}\label{sec_ex_smooth}

We take $\Omega = (0,1)^2$ and use the manufactured solution
$u(x,y)=x^2(1-x)^2y^2(1-y)^2$.

To compare the approximations given by the schemes \eqref{DPG1} and \eqref{DPG2},
we use piecewise constant functions on uniform meshes for $u_{i,h}$ and $\sigma_{i,h}$,
and traces of the reduced Hsieh--Clough--Tocher (HCT) functions
for both $\tu_{i,h}$ and $\tsigma_{i,h}$ ($i=1,2$). These HCT traces use piecewise cubic polynomials
for (standard) traces on edges and piecewise linear polynomials for normal derivatives on edges,
subject to the regularity of stemming from $H^2(\Omega)$-functions.
For the reduced HCT elements we refer to \cite{Ciarlet_78_IEE}, and the traces we use are described
in \cite{FuehrerHN_19_UFK}.

Figure~\ref{fig_smooth} presents the $L_2(\Omega)$ approximation errors for $u$ and $\sigma=\Delta u$
along with the corresponding residual $\eta$. The results for scheme \eqref{DPG1} are on the left
and for \eqref{DPG2} on the right. It appears that in both cases we have
an asymptotical behavior of $\|u-u_{i,h}\|\simeq\|\sigma-\sigma_{i,h}\|\simeq\eta=\OO(h)$.
This is expected for lowest order approximations of a smooth function.

\begin{figure}[htb]
  \begin{center}
    \includegraphics[width=0.45\textwidth]{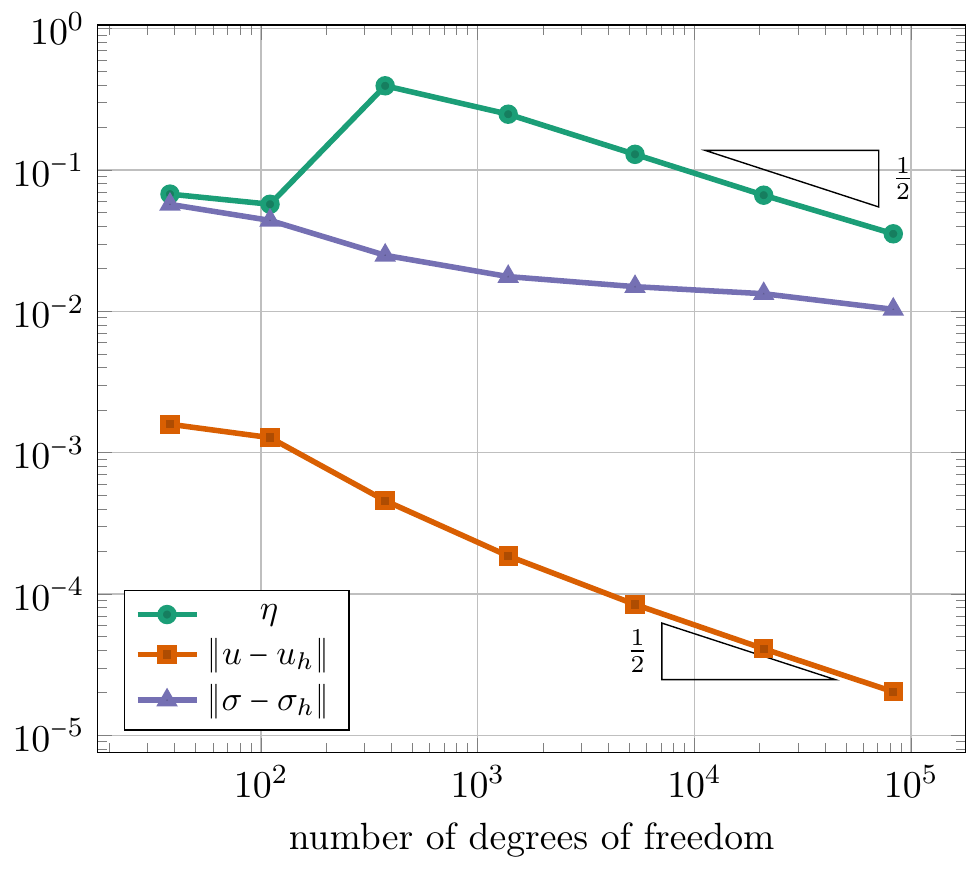}
    \includegraphics[width=0.45\textwidth]{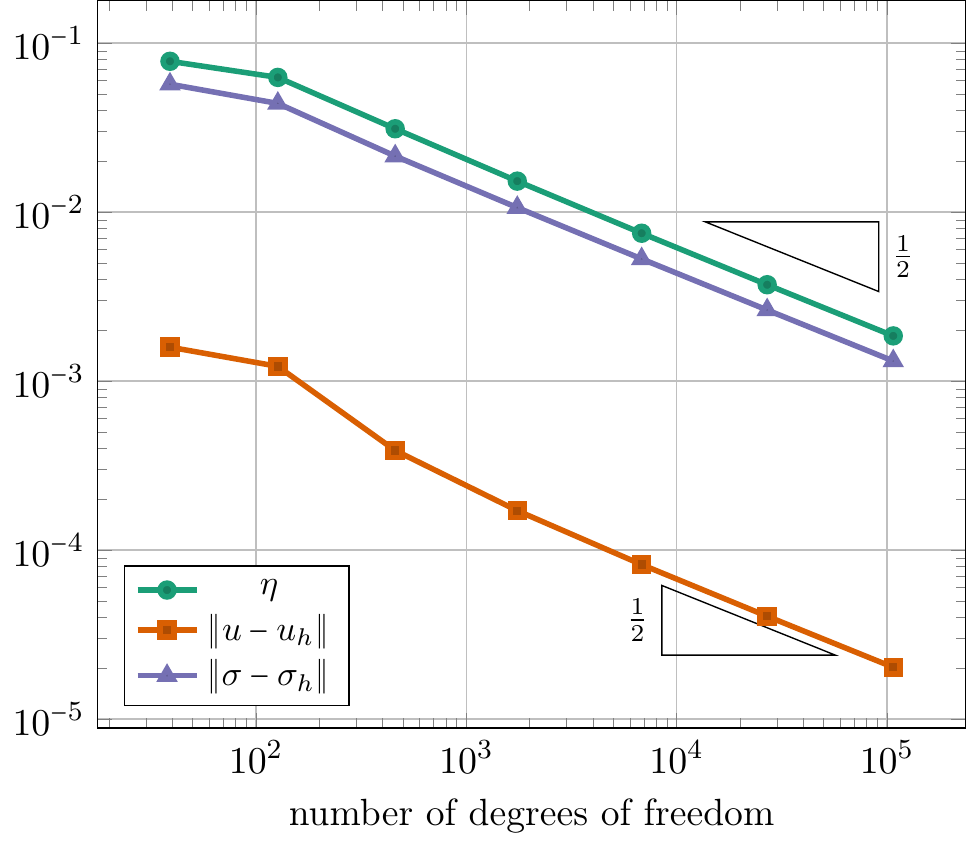}
  \end{center}
  \caption{Errors generated by schemes \eqref{DPG1} (left) and scheme \eqref{DPG2} (right)
          for the smooth example from \S\ref{sec_ex_smooth}.}
  \label{fig_smooth}
\end{figure}

\subsection{Example with singular solution}\label{sec_ex_sing}

The next example is taken from \cite{FuehrerHN_19_UFK}.
We consider the non-convex domain from Figure~\ref{fig_domain} with
reentrant corner at $(0,0)$. The outer angle at this corner is $\tfrac3{4}\pi$.
We take the manufactured solution
\begin{align*}
  u(r,\varphi) = r^{1+\alpha}(\cos( (\alpha+1)\varphi)+C \cos( (\alpha-1)\varphi))
\end{align*}
with polar coordinates $(r,\varphi)$ centered at the origin. It holds
\(
  \Delta^2 u = 0 =: f.
\)
For the boundary conditions we prescribe the values of $u|_\Gamma$ and $\nabla u|_\Gamma$.
The parameters $\alpha$ and $C$ are chosen such that $u$ and its normal derivative vanish on the boundary edges
that meet at the origin.
Here, we have $\alpha\approx 0.673583432147380$ and $C\approx 1.234587795273723$.
It holds $u\in H^{2+\alpha-\varepsilon}(\Omega)$ but $\Delta u\not\in H^1(\Omega)$.

\begin{figure}[htb]
  \begin{center}
    \includegraphics[width=0.4\textwidth]{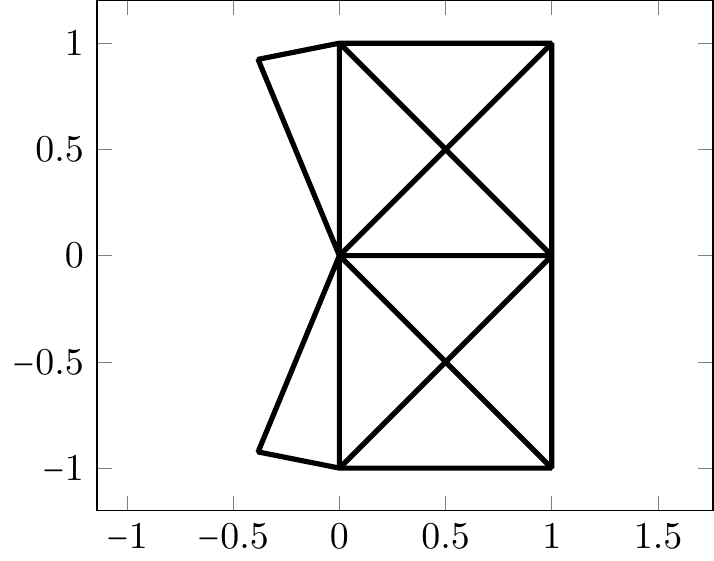}
  \end{center}
  \caption{The non-convex domain with initial mesh.}
  \label{fig_domain}
\end{figure}

The numerical results for the two schemes \eqref{DPG1} (on the left) and \eqref{DPG2} (on the right) are shown
in Figure~\ref{fig_sing}. As before, we plot the $L_2(\Omega)$-errors for $u$ and $\sigma=\Delta u$
along with the corresponding residual $\eta$.
In both cases the schemes converge at a low rate when using quasi-uniform meshes (curves without label ``adap''),
variant \eqref{DPG1} being extremely slow.
The rates exhibited by the second scheme are as expected by the regularity of $\sigma$.
However, scheme \eqref{DPG1} seems to suffer from the approximation
of $\tsigma_h$ by smooth $H^2$-traces. This is clearly not an efficient basis.
We can only claim convergence based on a density argument.

We have also used adaptive variants of both DPG schemes (curves with label ``adap'' in the same figures).
It turns out that the second scheme \eqref{DPG2} recovers its optimal rate of $\OO(\dim(\UU_{2,h})^{-1/2})$.
On the other hand, the residual $\eta$ and error $\|\sigma-\sigma_{1,h}\|$ of the first scheme
converge as slowly as before. Again, this seems to be caused by the inappropriate basis for $\tsigma_{1,h}$.
It is an open problem to construct discrete trace spaces that improve the convergence rate of scheme
\eqref{DPG1} for non-smooth solutions.

\begin{figure}[htb]
  \begin{center}
    \includegraphics[width=0.45\textwidth]{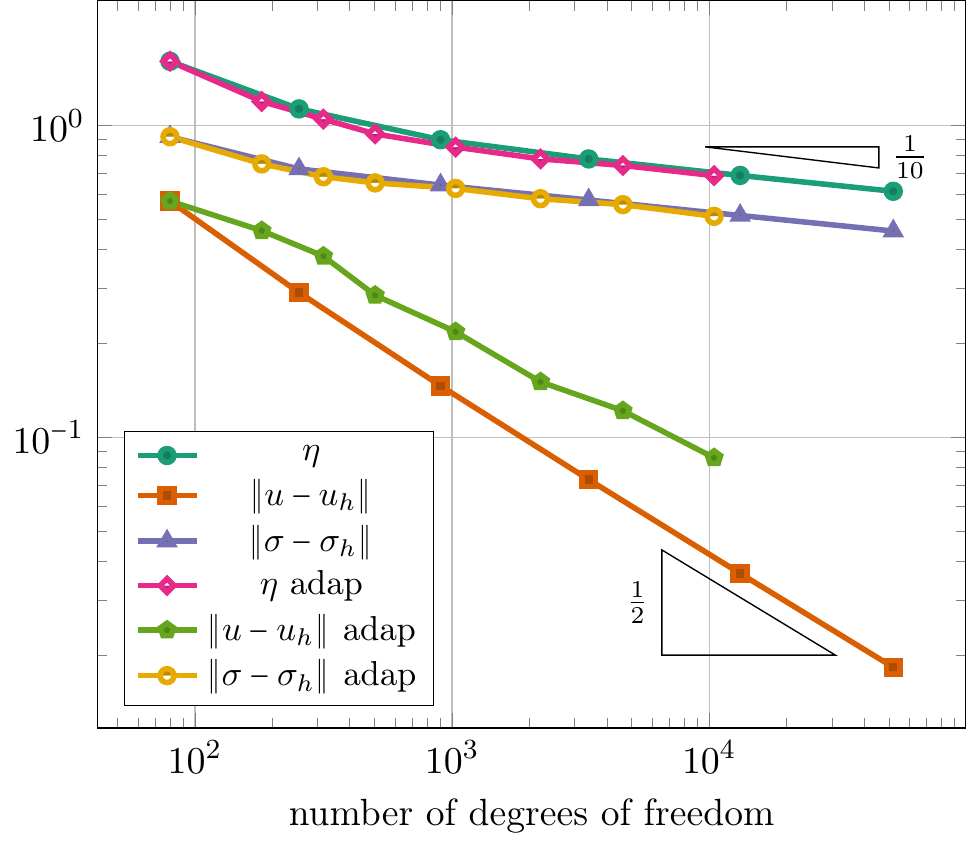}
    \includegraphics[width=0.45\textwidth]{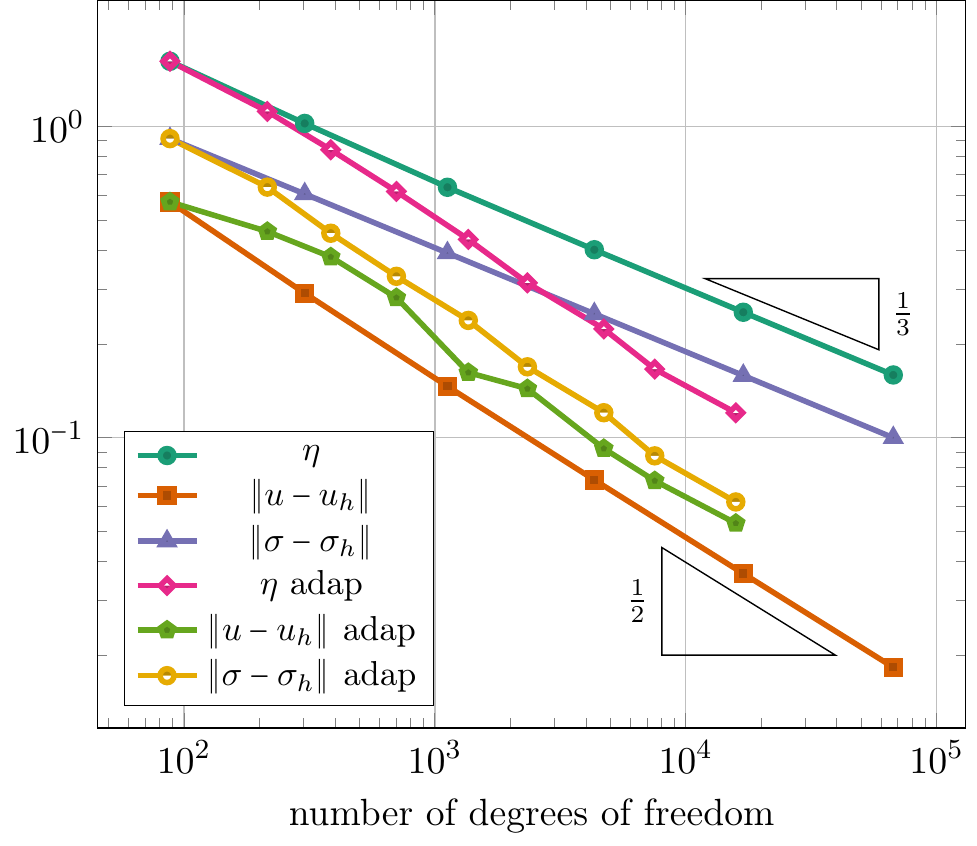}
  \end{center}
  \caption{Errors generated by schemes \eqref{DPG1} (left) and \eqref{DPG2} (right)
           for the singular example from \S\ref{sec_ex_sing}.}
  \label{fig_sing}
\end{figure}

\clearpage
\appendix
\section{The space $\tracetD{}(\HD{\Omega})$ is not closed in $H^2(\cT)'$.} \label{sec_app1}

We give an example for the two-dimensional case and conjecture that the result is true
in arbitrary space dimension $\di\ge 2$, replacing the Dirac delta below by distributions supported
on $(\di-1)$-simplexes.

We consider a domain in $\R^2$ with a vertex, for instance the open triangle $\Omega=T$
with vertices $(0,0)$, $(1,0)$, $(0,1)$. The construction that follows also works for
a Lipschitz domain where part of the boundary (connected, non-zero relative measure) is $C^2$,
considering an interior point of this smooth part instead of a vertex.

Let us recall the definition of the trace operator $\tracetD{T}$, cf.~\eqref{trtDT},
\[
   \tracetD{T}:\;
   \left\{\begin{array}{cll}
      \HD{T} & \to & H^2(T)',\\
      v & \mapsto & \tracetD{T}(v):\; z\mapsto \vdual{\Delta z}{v}_T-\vdual{z}{\Delta v}_T
   \end{array}\right..
\]
To show that $\tracetD{T}(\HD{T})$ is not closed we construct a sequence of smooth functions
$(v_\varepsilon)_\varepsilon$ (e.g., $\varepsilon=1/n$ with positive integer $n$) so that the corresponding
trace sequence $(\tracetD{T}(v_\varepsilon))_\varepsilon\subset\tracetD{T}(\HD{T})$ converges in $H^2(T)'$ to
the Dirac distribution at $(0,0)$. This distribution is an element of
the trace space $\bH^{-3/2,-1/2}(\partial T)=\traceDD{T}(\HdDiv{T})$, cf.~\cite{FuehrerHN_19_UFK},
but it is not the trace of an $\HD{T}$-function, cf.~Lemma~\ref{la_trtD} and Appendix~\ref{sec_app2}.

We remark that this construction does not contradict the closedness of the trace space
$\bHD{\partial T}=\traceD{T}(\HD{T})\subset \HD{T}'$ proved by Lemma~\ref{la_tr_unity}.
Indeed, Dirac distributions at boundary points $e$ are not elements of
$ \HD{T}'$ since, e.g., $w=\log|\cdot,e|$ satisfies
$w\in\HD{T}$ (because $\Delta w=0$), but its value at $e$ is not controlled.

We start by considering the mollifier type functions
\begin{align*}
  \phi_\varepsilon(t) := \begin{cases}
    C\frac1{\varepsilon} e^{-\varepsilon^2/(\varepsilon^2-t^2)},  & t\in[0,\varepsilon), \\
    0, & \text{else},
  \end{cases}
\end{align*}
where $C>0$ is chosen such that $\int_0^1 \phi_\varepsilon(t)\,dt = 1/2$.
Note that $C$ is independent of $\varepsilon$.

In the following let us denote $I=(0,1)$ and $I_\varepsilon=(0,\varepsilon)$.
We need two technical results.

\begin{lemma} \label{lem_distributionZero} We have that
  \begin{align*}
    \|t\mapsto t\phi_\varepsilon(t)\|_{I} \to 0 \quad\text{as}\quad \varepsilon\to 0.
  \end{align*}
  (Here and in the following, $\|\cdot\|_I$ denotes the $L_2(I)$-norm.)
\end{lemma}

\begin{proof}
  Since $\phi_\varepsilon(t)$ takes its maximum at $t=0$ we can bound
  \begin{align*}
    \|t\mapsto t\phi_\varepsilon(t)\|_{I}^2
    \leq C^2e^{-2} \frac1{\varepsilon^2} \int_0^\varepsilon t^2 \,dt 
    = \OO(\varepsilon).
  \end{align*}
\end{proof}

\begin{lemma} \label{lem_embedding}
  Let $\varepsilon>0$ and $v\in H^1(I_\varepsilon)$ with $v(0) = 0$. It holds the bound
  \begin{align*}
    \|v\|_{L_\infty(I_\varepsilon)} \lesssim \varepsilon^{1/2}\|v'\|_{I_\varepsilon}
  \end{align*}
  with hidden constant independent of $\varepsilon$ and $v$.
\end{lemma}

\begin{proof}
The statement follows in the standard way by using
the continuous embedding $H^1(I)\hookrightarrow C^0(\bar{I})$, a Poincar\'e inequality, and
scaling arguments.
\end{proof}

Now, using $\phi_\varepsilon$, we define $v_\varepsilon\in C^\infty(T)$ by
\begin{align*}
  v_\varepsilon(x,y) := -(x+y) \phi_\varepsilon(|(x,y)|) = -(x+y) \begin{cases}
    C\frac1{\varepsilon} e^{-\varepsilon^2/(\varepsilon^2-x^2-y^2)},  & |(x,y)| < \varepsilon, \\
    0, & \text{else}.
  \end{cases}
\end{align*}

\begin{theorem}
  Let $\delta\in \bH^{-3/2,-1/2}(\partial T)\subset H^2(T)'$ denote the Dirac distribution supported at $(0,0)$,
that is, $\dual{\delta}{z}_{\partial T} = z(0,0)$ for $z\in H^2(T)$. It holds
\begin{align*}
    \tracetD{T}(v_\varepsilon) \to \delta \quad(\varepsilon\to 0)\quad\text{in}\quad H^2(T)'.
\end{align*}
\end{theorem}

\begin{proof}
Since $v_\varepsilon$ is smooth we can represent its trace as
\[
   \dual{\tracetD{T}(v_\varepsilon)}{z}_{\partial T}
   =
   \dual{\partial_\nn z}{v_\varepsilon}_{\partial T}-\dual{z}{\partial_\nn v_\varepsilon}_{\partial T}
   \quad\forall z\in H^2(T).
\]
To obtain a representation of $\partial_\nn v_\varepsilon$ we note that for $\varepsilon<1/2$,
$v_\varepsilon$ and its derivatives vanish on the edge spanned by the nodes $(1,0)$, $(0,1)$.
Second, we have that
\begin{align*}
  \nabla v_\varepsilon(x,y) = \frac{C}\varepsilon e^{-\varepsilon^2/(\varepsilon^2-x^2-y^2)} 
  \begin{pmatrix} -1 \\ -1 \end{pmatrix} + \frac{C(x+y)}{\varepsilon} e^{-\varepsilon^2/(\varepsilon^2-x^2-y^2)}
  \frac{\varepsilon^2}{(\varepsilon^2-x^2-y^2)^2} \begin{pmatrix} 2x \\ 2y \end{pmatrix}.
\end{align*}
Let $E:= (0,1)\times\{0\}$. Then, $\nn_E = (0,-1)^\top$ and
\begin{align} \label{eq_dnurepres}
  \partial_{\nn_E} v_\varepsilon|_E (t) = \nn_E\cdot \nabla v_\varepsilon(t,0) 
  = \frac{C}{\varepsilon} e^{-\varepsilon^2/(\varepsilon^2-t^2)} = \phi_\varepsilon(t).
\end{align}
Here, $t=x$ is the (local) arc length of $E$ starting at $(0,0)$.
Similarly, we calculate $\partial_{\nn_{E'}} v_\varepsilon|_{E'}(t)=\phi^-_\varepsilon(t):=\phi_\varepsilon(1-t)$
where $t=1-y$ is the arc length of $E'$ starting at $(0,1)$.

Now, for $z\in H^2(T)$, we find that
  \begin{align*}
    \dual{\tracetD{T}(v_\varepsilon)}{z}_{\partial T}
    = \dual{v_\varepsilon}{\partial_\nn z}_{\partial T} - \dual{\partial_\nn v_\varepsilon}{z}_{\partial T}
    = \dual{v_\varepsilon}{\partial_\nn z}_E + \dual{v_\varepsilon}{\partial_\nn z}_{E'}
    - \dual{\phi_\varepsilon}{z}_E - \dual{\phi^-_\varepsilon}{z}_{E'}
  \end{align*}
with $L_2(E)$-duality $\dual{\cdot}{\cdot}_E$, and correspondingly for $E'$.
  Note that 
  \begin{align*}
    \dual{v_\varepsilon}{\partial_\nn z}_E
    \leq \|v_\varepsilon\|_E \|\grad z\|_E
    \lesssim \|v_\varepsilon\|_E \norm{z}{2,T}
    = \|t\mapsto t\phi_\varepsilon(t)\|_I \norm{z}{2,T}.
  \end{align*}
  We obtain the very same estimate replacing $E$ by $E'$.
  Lemma~\ref{lem_distributionZero} then proves that
  \begin{align}\label{eq_estueps}
    \sup_{0\not=z\in H^2(T)} \frac{\dual{v_\varepsilon}{\partial_\nn z}}{\norm{z}{2,T}} \lesssim
    \|t\mapsto t\phi_\varepsilon(t)\|_I \to 0 \quad (\varepsilon \to 0).
  \end{align}
  To rewrite and estimate the term $\dual{\partial_\nn v_\varepsilon}{z}_E$ we use the
  representation~\eqref{eq_dnurepres} and the fact that $\int_0^1 \phi_\varepsilon(t)\,dt = 1/2$.
  This shows that
  \begin{align*}
    \dual{\partial_\nn v_\varepsilon}{z}_E = \dual{\phi_\varepsilon}{z(\cdot,0)}_I = 
    \frac12 z(0,0) + \dual{\phi_\varepsilon}{z(\cdot,0)-z(0,0)}_I
    \quad (z\in H^2(T)).
  \end{align*}
  Analogously it holds
  \begin{align*}
    \dual{\partial_\nn v_\varepsilon}{z}_{E'} = \dual{\phi^-_\varepsilon}{z(0,1-\cdot)}_I = 
    \frac12 z(0,0) + \dual{\phi^-_\varepsilon}{z(0,1-\cdot)-z(0,0)}_I
    \quad (z\in H^2(T)).
  \end{align*}
  The last term in the latter two estimates can be estimated as follows (we only consider the first one).
  Since $\mathrm{supp}(\phi_\varepsilon|_I) = [0,\varepsilon]$, 
  Lemma~\ref{lem_embedding} and a trace inequality show that
  \begin{align*}
    |\dual{\phi_\varepsilon}{z(\cdot,0)-z(0,0)}_I|
    &\leq
    \|\phi_\varepsilon\|_{L_1(I_\varepsilon)} \|z(\cdot,0)-z(0,0)\|_{L_\infty(I_\varepsilon)}
    \\ 
    &\lesssim \varepsilon^{1/2} \|\partial_x z\|_{I_\varepsilon} 
    \leq \varepsilon^{1/2} \|\partial_x z\|_I \lesssim \varepsilon^{1/2} \|z\|_{2,T}.
  \end{align*}
  Now, using the delta distribution $\dual{\delta}{z}_{\partial T}=z(0,0)$, we therefore obtain
  for any $z\in H^2(T)$
  \begin{align*}
    |\dual{\delta}{z}_{\partial T}-\dual{\partial_\nn v_\varepsilon}z_{\partial T}| 
    &= |\dual{\phi_\varepsilon}{z(\cdot,0)-z(0,0)}_I +
    \dual{\phi^-_\varepsilon}{z(0,1-\cdot)-z(0,0)}_I|
    \lesssim \varepsilon^{1/2}\|z\|_{2,T}.
  \end{align*}
  This bound, together with~\eqref{eq_estueps}, shows that
  \begin{align*}
    \sup_{0\not=z\in H^2(T)} \frac{\dual{\delta-\tracetD{T}(v_\varepsilon)}z_{\partial T}}{\|z\|_{2,T}}
    \leq
    \sup_{0\not=z\in H^2(T)}
    \frac{\dual{\delta}z_{\partial T} - \dual{\partial_\nn v_\varepsilon}z_{\partial T}}{\|z\|_{2,T}}
    +
    \sup_{0\not=z\in H^2(T)}
    \frac{\dual{v_\varepsilon}{\partial_\nn z}_{\partial T}}{\|z\|_{2,T}} \to 0
  \end{align*}
  when $\varepsilon\to 0$. This finishes the proof.
\end{proof}

\section{The Dirac mass is not an element of $\tracetD{}(\HD{\Omega})$.} \label{sec_app2}

  The following argument is essentially the observation that fundamental solutions to the Laplacian
  (in any space dimension $\ge 2$) are not bounded. For illustration we show details for the case $\di=2$.
  Without loss of generality we assume that $\Omega$ is the upper half of a circle with center
  $x_0 = (0,0)$ and radius 1. (A smoothness of the boundary apart from Lipschitz continuity is not
  needed in our construction.)
  We define the points $x_n = \nn(x_0)\tfrac1n = (0,-\tfrac1n) \notin\overline\Omega$ and consider
  the sequence
  \begin{align*}
    (v_n)_n\quad\text{with}\quad
    v_n := \log|x_n-\cdot| \in H^2(\Omega).
  \end{align*}
  Since this sequence converges pointwise in $\Omega$ to $v:=\log|\cdot|$ and is bounded in $L_2(\Omega)$,
  it converges in $L_2(\Omega)$ to $v$. It also converges in $\HD{\Omega}$ to $v$ because
  $\Delta v_n=\Delta v=0$ in $\Omega$.

  Now we argue by contradiction. Suppose there exists $\sigma\in\HD{\Omega}$ with
  $\tracetD{\Omega}(\sigma) = \delta_{x_0}$ (the Dirac delta supported at $x_0$), i.e.,
  $\dual{\tracetD{\Omega}(\sigma)}{z}_\Gamma = z(x_0)$ for all $z\in H^2(\Omega)$.
  
  Since $\sigma, v\in\HD{\Omega}$ the value $\dual{\traceD{\Omega}(\sigma)}v_\Gamma$ has to be finite.
  Moreover, since $v_n\to v$ in $\HD{\Omega}$ as $n\to\infty$,
  $\dual{\traceD{\Omega}(\sigma)}{v_n}_\Gamma \to \dual{\traceD{\Omega}(\sigma)}{v}_\Gamma < \infty$.
  However, since $v_n\in H^2(\Omega)$ we conclude that
  \begin{align*}
    \dual{\traceD{\Omega}(\sigma)}{v_n}_\Gamma
    =
    \dual{\tracetD{\Omega}(\sigma)}{v_n}_\Gamma = v_n(x_0) \to \infty \text{ as } n\to \infty.
  \end{align*}
  This contradicts $\dual{\traceD{\Omega}(\sigma)}{v}_\Gamma < \infty$.

\section{On the regularity of solutions to the bi-Laplace problem.} \label{sec_reg}

We note a regularity result for the bi-Laplacian. Generally, a solution $u$ to problem \eqref{prob}
is a priorily sought in $H^2_0(\Omega)$. In this respect, \eqref{eq_Dt} is a fundamental relation
to show the ellipticity (coercivity) of the induced bilinear form $\vdual{\Delta\cdot}{\Delta\cdot}$.
On the other hand, for a right-hand side function $f\in L_2(\Omega)$, the proof of Theorem~\ref{thm_stab1}
shows that this regularity is automatically satisfied.

\begin{theorem} \label{thm_ref}
Let $u\in\HDz{\Omega}$ with $\Delta u\in\HD{\Omega}$. It holds $u\in H^2_0(\Omega)$ so that,
in particular, $\|\Delta u\|=\|\Grad\grad u\|$.
\end{theorem}

\begin{proof}
By assumption, $u\in\HDz{\Omega}$ satisfies $f:=\Delta^2 u\in L_2(\Omega)$. The proof of Theorem~\ref{thm_stab1}
shows that this problem has a unique solution $u\in H^2_0(\Omega)$. Then,
$\|\Delta u\|=\|\Grad\grad u\|$ holds by \eqref{eq_Dt}.
\end{proof}

\bibliographystyle{siam}
\bibliography{/home/norbert/tex/bib/heuer,/home/norbert/tex/bib/bib}
\end{document}

%% file: header.tex
\newtheorem{theorem}{Theorem}
\newtheorem{lemma}[theorem]{Lemma}

\newtheorem{prop}[theorem]{Proposition}
\newtheorem{remark}[theorem]{Remark}
\theoremstyle{definition} 

\newcommand{\<}{\langle{}}
\renewcommand{\>}{\rangle}

\newcommand{\ip}[2]{\llangle#1\hspace*{.5mm},#2\rrangle}
\newcommand{\dual}[2]{\<#1\hspace*{.5mm},#2\>}
\newcommand{\vdual}[2]{(#1\hspace*{.5mm},#2)}

\newcommand{\norm}[2]{\|#1\|_{#2}}

\newcommand{\diam}{\mathrm{diam}}

\newcommand{\wat}{\widehat}
\newcommand{\transp}{\mathsf{T}}

\def\Grad{\boldsymbol{\varepsilon}}

\def\Div{{\rm\bf div\,}}

\def\grad{\nabla}

\def\pwDelta{\Delta_\cT}

\def\TTheta{\mathbf{\Theta}}

\newcommand{\LL}{\ensuremath{\mathbb{L}}}

\def\tq{\wat{\boldsymbol{q}}}

\def\tu{\wat{\boldsymbol{u}}}
\def\tsigma{\wat{\boldsymbol{\sigma}}}
\def\tv{\wat{\boldsymbol{v}}}
\def\tz{\wat{\boldsymbol{z}}}

\newcommand{\UU}{\ensuremath{\mathcal{U}}}
\newcommand{\VV}{\ensuremath{\mathcal{V}}}

\newcommand{\vv}{\ensuremath{\mathbf{v}}}
\newcommand{\uu}{\ensuremath{\mathbf{u}}}
\newcommand{\ww}{\ensuremath{\mathbf{w}}}

\newcommand{\deltav}{\delta\!v}
\newcommand{\deltaz}{\delta\!z}

\newcommand{\diagtensor}{\mathbb{D}}

\newcommand{\traceD}[1]{\mathrm{tr}_{#1}^{\Delta}} 
\newcommand{\traceDD}[1]{\mathrm{tr}_{#1}^{\mathrm{dDiv}}} 
\newcommand{\traceDt}[1]{\mathrm{tr}_{#1}^{\Delta,2}}
\newcommand{\tracetD}[1]{\mathrm{tr}_{#1}^{2,\Delta}} 

\newcommand{\bH}{\ensuremath{\mathbf{H}}}
\newcommand{\bHD}[1]{\ensuremath{\mathbf{H}^\Delta(#1)}}
\newcommand{\bHDzz}[1]{\ensuremath{\mathbf{H}_{00}^\Delta(#1)}}
\newcommand{\bHDt}[1]{\ensuremath{\mathbf{H}^{\Delta,2}(#1)}}
\newcommand{\bHDtzz}[1]{\ensuremath{\mathbf{H}_{00}^{\Delta,2}(#1)}}

\newcommand{\cD}{\ensuremath{\mathcal{D}}}

\newcommand{\trD}[1]{{\mathrm{\Delta},#1}}
\newcommand{\trt}[1]{{2,#1}}
\newcommand{\trddiv}[1]{{\mathrm{dDiv},#1}}

\def\div{{\rm div\,}}

\newcommand{\ttt}{{\rm T}}

\newcommand{\HD}[1]{{H(\Delta,#1)}}
\newcommand{\HdDiv}[1]{{H(\div\Div\!,#1)}}
\newcommand{\HDz}[1]{{H_0(\Delta,#1)}}

\newcommand{\di}{d}
\newcommand{\R}{\ensuremath{\mathbb{R}}}

\newcommand{\nn}{\ensuremath{\mathbf{n}}}

\newcommand{\cT}{\ensuremath{\mathcal{T}}}

\newcommand{\cS}{\ensuremath{\mathcal{S}}}

\newcommand{\OO}{\ensuremath{\mathcal{O}}}
